\documentclass[11pt,reqno]{amsart}
\usepackage{amsmath,amsfonts,amsthm,amssymb,geometry,graphics,enumerate,newlfont,latexsym,tabularx,shapepar}
\usepackage[all,2cell,dvips]{xy}\UseAllTwocells\SilentMatrices
\usepackage{tikz-cd}
\usetikzlibrary{matrix,arrows,decorations.pathmorphing}

\newtheorem{thm}{Theorem}[section]
\newtheorem*{theorem*}{Theorem}
\newtheorem*{corollary*}{Corollary}
\newtheorem*{proposition*}{Proposition}
\newtheorem{defin}[thm]{Definition}
\newtheorem{lemme}[thm]{Lemma}

\newtheorem{prop}[thm]{Proposition}
\newtheorem{asum}{Assumption}
\newtheorem{cor}[thm]{Corollary}

\newtheorem{remark}{Remark}

\DeclareMathOperator{\diam}{diam}
\DeclareMathOperator{\dist}{dist}
\DeclareMathOperator{\lip}{Lip}

\DeclareMathOperator{\dime}{dim}

\DeclareMathOperator{\hes}{Hess}

\DeclareMathOperator{\dvol}{dvol}
\DeclareMathOperator{\vol}{vol}
\DeclareMathOperator{\h}{H}
\DeclareMathOperator{\R}{R}

\DeclareMathOperator{\tr}{trace}
\DeclareMathOperator{\sect}{sec}
\DeclareMathOperator{\inj}{injrad}
\DeclareMathOperator{\aut}{Aut}

\DeclareMathOperator{\Id}{Id}
\DeclareMathOperator{\e}{e}

\DeclareMathOperator{\spt}{spt}
\DeclareMathOperator{\sing}{sing}
\begin{document}
\title[ Harmonic Maps]{Convergence of Harmonic Maps}

\author{Zahra Sinaei} 

\begin{abstract}
In this paper we prove a compactness theorem for sequences of harmonic maps which are defined on converging sequences of Riemannian manifolds.
\end{abstract}
\maketitle
%%%%%%%%%%%%%%%%%%%%%%%%%%%%%%%%%%%%%%%%%%%%%%%%%%%%%%%%%%%%%%%%%%%%%%%%%%%%%%%%%%%%%%%%%%%%%%%%%%%%%%%%%%%
%%%%%%%%%%%%%%%%%%%%%%%%%%%%%%%%%%%%%%%%%%%%%%%%%%%%%%%%%%%%%%%%%%%%%%%%%%%%%%%%%%%%%%%%%%%%%%%%%%%%%%%%%%%
Harmonic maps  are critical points of the energy functional defined on the space of maps between Riemannian manifolds. This theory was developed by J. Eells and H. Sampson \cite{E64} in the 1960s. The notion of harmonic maps on smooth metric measure spaces was introduced by Lichnerowicz in \cite{L69}.  Harmonic maps betweens singular spaces have been studied since the early 1990s in the works of Gromov-Schoen in \cite{GS92} and Korevaar-Schoen in \cite{KS93}.  Eells and Fuglede describe the application of the methods of \cite{KS93} to the study of maps between polyhedra \cite{F01}.

A smooth metric measure space is a triple $(M,g,\Phi \dvol_M)$, where $(M,g)$ is an $n$-dimensional Riemannian manifold, $\dvol_M$ denotes the corresponding  Riemannian volume element on $M$, and $\Phi$ is a smooth positive function on $M$. These spaces have been used extensively in geometric analysis and they arise as smooth collapsed measured Gromov-Hausdorff limits in the works of Cheeger-Colding \cite{CC97, C00, CC00}, Fukaya  \cite{Fuk87} and Gromov \cite{Gr81}. They have been studied recently by Morgan \cite{M05}. See also works of Lott \cite{L03}, Qian \cite{Q97}, Fang-Li-Zhang \cite{F09}, Wei-Wylie \cite{W09}, Wu  \cite{W10}, Su-Zhang \cite{Z11} and  Munteanu-Wang \cite{M11}.

In this paper, we are going to study the behavior of harmonic maps under convergence.  Let ${\cal{M}}(n,D)$ denote the set of all compact Riemannian manifolds  $(M,g)$ such that $\dime(M)=n$, $\diam(M)<D$, and the sectional curvature $\sect_g$ satisfies $|\sect_g|\leq 1$, equipped with the measured Gromov-Hausdorff topology. Let $(M_i,g_i,\dvol_{M_i})$ in $\mathcal{M}(n,D)$ be a sequence of manifolds which converges to a smooth metric measure space $(M,g,\Phi\dvol_M)$. Suppose $f_i:(M_i,g_i)\rightarrow (N,h)$ is a sequence of harmonic maps. We are interested in knowing under what circumstances the $f_i$ converge to a harmonic map $f$ on the smooth metric measure space $(M,g,\Phi\dvol_M)$.

 When a sequence of manifolds $(M_i,g_i)$ in $\mathcal{M}(n,D)$ converges to a metric space $X$, according to Fukaya \cite{F88}, $X$ is a quotient space $Y\slash O(n)$, where $Y$ is a smooth manifold. Indeed $Y$ is the limit point of the sequence of frame bundles, $F(M_i)$, over the manifolds $M_i$ and $X$ has the structure of a Riemannian polyhedron $(X,g_X,\Phi_X\mu_g)$  where $\mu_g$ is the Riemannian volume element related to the metric $g_X$ on $X$.

We state the main result of this paper which is a compactness theorem for sequences of harmonic maps.
\begin{thm}\label{main}. Let $(M_i,g_i)$ be a sequence of smooth Riemannian manifolds in ${\mathcal{M}}(n,D)$ which converges to a metric measure space $(X,g,\Phi  \mu_g)$ in the measured Gromov-Hausdorff Topology. Suppose $(N,h)$ is a compact Riemannian manifold. Let $f_i:(M_i,g_i)\rightarrow (N,h)$ be a sequence of harmonic maps such that $\|e_{g_i}(f_i)\|_{L^{\infty}}<C$, where $\|e_{g_i}(f_i)\|_{L^{\infty}}$ is the  $L^{\infty}$-norm of the energy density of the map $f_i$ and $C$ is a constant independent of $i$. Then $f_i$ has a subsequence  which converges to a map  $f:(X,g,\Phi\mu_g)\rightarrow (N,h)$, and this map is a harmonic map in ${\mathcal{H}}^1((X,\Phi \mu_g),N)$.
\end{thm}
 By ${\mathcal{H}}^1(X,N)$ we mean
$$\{f\in{\mathcal{H}}^1(X,\mathbb R^q)~\big{|}~ f(x)\in N~ \text{for almost all}~ x\in M \},$$
where ${\mathcal{H}}^1(X,\mathbb R^q)$ is the standard Sobolev space and $N$ is isometrically embedded in $\mathbb R^q$. In this work we use the notations  ${\mathcal{H}}^1$ and $W^{1,2}$ interchangeably. For the notion of convergence of maps we refer the reader to the Definition \ref{convergence-maps}.

The rest of this paper is organized as follows. In the first section we introduce our main notations and preliminary results needed for the rest of this paper. In the second section, we prove Theorem \ref{main}. We divide the proof into three cases. In Subsection \ref{casei} we consider the non-collapsing case, Proposition \ref{CaseI}. Moreover using the regularity results for harmonic maps in the work of Schoen and Lin \cite{S83,L99} we  study Theorem \ref{main} under less restrictive assumption of uniform boundedness of the energy  of the maps  $f_i$ (see Propositions \ref{bndenrg}, \ref{bndenrg1}).  In subsection \ref{caseii} we consider the  case of collapsing to a Riemannian manifold, Proposition \ref{CaseII}. As a preliminary step we prove the result under  some regularity assumption on the metrics $g_i$,  see Proposition \ref{jadid}. The general case is considered in subsection \ref{caseiii}.  The  Appendix   is devoted to the study of convergence of the tension fields of the maps $f_i$ under the assumptions of Proposition \ref{jadid}.
\section*{Acknowledgement}
This work is part of my Ph.D. dissertation. I thank my advisor Professor
Marc Troyanov for his guidance and support in the completion of this work.
I also thank Professors Buser, Naber, and Wenger for their reading of this document
and  their comments and suggestions.
%%%%%%%%%%%%%%%%%%%%%%%%%%%%%%%%%%%%%%%%%%%%%%%%%%%%%%%%%%%%%%%%%%%%%%%%%%%%%%%%%%%%%%%%%%%%%%%%%%%%%%%%%%%%%%%%%%%%%%%%%%%%%%%%%%%%%%%%%
%%%%%%%%%%%%%%%%%%%%%%%%%%%%%%%%%%%%%%%%%%%%%%%%%%%%%%%%%%%%%%%%%%%%%%%%%$$%%%%%%%%%%%%%%%%%%%%%%%%%%%%%%%%%%%%%%%%%%%%%%%%%%%%%%%%%%%%%%%
%\bibliographystyle{alpha}
%\bibliography{biblio}
\section{Background}
\subsection{Harmonic maps}
 In this subsection, we first recall the definition of  weakly harmonic maps on  smooth metric measure spaces. We then briefly review this concept on Riemannian polyhedra. At the end we present some theorems and lemmas that we need in this paper.
 Let $(N,h)$ be a compact Riemannian manifold and $I$  an isometric embedding $I:N\rightarrow\mathbb R^q$. Since $I(N)$ is a smooth, compact submanifold of $\mathbb R^q$, there exists a number $\kappa>0$ such that the neighborhood
\begin{align*}
U_{\kappa}(N) = \{y\in \mathbb R^q: \dist(y,N)<\kappa\}
\end{align*}
has the following property: for every $y$ in $U_{\kappa}(N)$ there exists a unique point
$\pi_N(y)\in N$ such that
\begin{align*}
|y-\pi_N(y)|=\dist(y,N)
\end{align*}
 The map $\pi_N: U_{\kappa}(N)\rightarrow N$ defined as above is called the {\it{nearest point projection}} onto $N$.

The $\hes{\pi_N}$ defines an element in $\Gamma(TN^*\otimes TN^*\otimes TN^{\bot})$ which coincides with the second fundamental form of $I:N\rightarrow \mathbb R^q$ up to a negative sign
\begin{align*}
\langle\hes{\pi_N}(y)(X,Y),\eta\rangle=-\langle \nabla_Y\eta,X\rangle
\end{align*}
where $X$ and $Y$ are in $TN$, $y$ in $N$ and $\eta$ in  $TN^{\bot}$ (see \S3 in Moser \cite{MO05}).

 A map $f:(M,g,\Phi\dvol_M)\rightarrow(N,h)$, belonging to $\cal{H}^1_{loc}((M,\Phi \dvol_M),N)$ is called {\em{weakly harmonic}} if and only if
\begin{align}\label{EL}
\Delta I\circ f-\Pi(f)(df,df)+d I\circ f(\nabla \ln(\Phi))=0
\end{align}
in the weak sense. Here
\begin{align}\label{scndfnd}
\Pi(f)(df,df)= \tr~\hes(\pi_N)(I\circ f)(dI\circ f,dI\circ f),
\end{align}
or in coordinates
\begin{align*}
\Pi(f)(df,df)=\sum g^{ij}\frac{\partial^2\pi_N^{A}}{\partial z^B\partial z^C}\frac {\partial {f}^B  }{\partial x^i}\frac {\partial {f}^C  }{\partial x^j}.
\end{align*}
For $f:(M^n,g)\rightarrow (N^m,h)$ and $\eta:M\rightarrow \mathbb R^q$, we define
 \begin{align}\label{khi}
 \Xi_g(f,\eta)=\langle dI\circ f, d\eta\rangle -\langle\Pi(f)(df,df),\eta\rangle.
 \end{align}
We explain now what we mean by harmonic maps on Riemannian polyhedra. Following Eells-Fuglede \cite{F01} on an admissible Riemannian polyhedron $X$, a continuous weakly harmonic map $u:(X,g,\mu_g)\rightarrow (N,h)$ is  of class $\cal{H}^1_{loc}(X,N)$ and satisfies:  for any chart $\eta: V \rightarrow \mathbb R^n$ on $N$ and any open set $U \subset u^{-1}(V)$ of
compact closure in $X$, the equality 
\begin{equation}\label{EL2}
\int_Ug(\nabla\lambda,\nabla u^k) ~d\mu_g=\int_U\lambda (\Gamma_{\alpha\beta}^k\circ u)g(\nabla u^{\alpha},\nabla u^{\beta})~d\mu_g
\end{equation}
holds for every $k = 1,\ldots, n$ and every bounded function $\lambda\in \cal{H}^1_{0}(U)$. Here $\Gamma_{\alpha\beta}^k$  denote the Christoffel symbols on $N$.
 Similarly on a polyhedron $X$ with a measure $\Phi \mu_g$,  a continuous weakly harmonic map is a map in $ \cal{H}^1_{loc}((X,\Phi\mu_g),N)$ which satisfies  equation (\ref{EL2}) with $\Phi d\mu_g$ in place of $d\mu_g$. When the target is compact a continuous map $f$ on an admissible Riemannian polyhedron is harmonic if and only if it satisfies  (\ref{EL}) weakly.

\begin{thm}[Moser \cite{MO05}, Theorem 3.1]\label{weak}
Let $f\in {\cal{H}}^1(U,N)\cap C^0(U,N)$ be a weakly harmonic map, where $U$ is an open domain in $\mathbb R^n$. Then $f$ is smooth.
\end{thm}
The energy functional is lower semi continuous, and we have
\begin{lemme}[Xin \cite{X96}]\label{xin}
Let $S\subset {\cal{H}}^1(M,N)$ be such that the energy functional  is bounded  on $S$ and $S$ is closed under weak limits. Then $S$ is sequentially compact.
\end{lemme}
Now we recall some regularity results for harmonic maps from \cite{S83} and \cite{L99}. Let $M$ and $N$ be compact Riemannian  manifolds. Define
\begin{align*}
\mathcal{F}_{\Lambda}=\{u\in C^{\infty}(M,N): ~u ~\text{is harmonic and}~ E(u)\leq \Lambda\}.
\end{align*}
We have the following results.
\begin{thm}[Schoen \cite{S83}]\label{reg-sch}
Let $M$ and $N$ be compact Riemannian manifolds. Any map $u$ in the weak closure of $\mathcal{F}_{\Lambda}$ is smooth and harmonic outside a relatively closed singular set of locally finite Hausdorff $(n-2)$-dimensional measure.
\end{thm}
\begin{remark}[Schoen \cite{S83}, Lin \cite{L99}]\label{lin}\normalfont
Let $u_i$ be a sequence in $\mathcal{F}_{\Lambda}$. Then there exists a subsequence which converges weakly to some $u$ in $\mathcal{H}^1(M,N)$. Define
\begin{align*}
\Sigma=\bigcap_{r>0}\left\{x\in M, ~\liminf_{i\rightarrow\infty}r^{2-n}\int_{B_r(x)}e(u_i)\geq \epsilon_0\right\}
\end{align*}
where $\epsilon_0=\epsilon_0(n,N)>0$ is a constant independent of $u_i$ as in Theorem 2.2 in \cite{S83}. If we consider a sequence of Radon measures $\mu_i=|du_i|^2dx$, without loss of generality we may assume $\mu_i\rightharpoonup\mu$ weakly as Radon measures. By Fatou's lemma, we may write
\begin{align*}
\mu=|d u|^2dx+\nu
\end{align*}
 for some non-negative Radon measure $\nu$. We can show that $\Sigma=\spt\nu\cup\sing u$ and $\nu$ is absolutely continuous with respect to $H^{n-2}|_{\Sigma}$. Therefore $u_i$  converges strongly in $\mathcal{H}^1(M,N)$ to $u$ if and only if $|du_i|^2 dx\rightharpoonup|du|^2 dx$ weakly, if and only if $\nu=0$, if and only if $H^{n-2}(\Sigma)=0$, if and only if there is no smooth non-constant harmonic map from 2-sphere $\mathbb S^2$  into $N$, e.g.  negatively curved manifolds. See Lemma 3.1 in \cite{L99} for a complete discussion.
\end{remark}

The following reduction theorem shows the relation between the tension fields of equivariant harmonic maps under Riemannian submersions.
\begin{thm}[Xin \cite{X96}, Theorem 6.4]\label{reduction}
 Let $\pi_1:E_1\rightarrow M_1$ and $\pi_2:E_2\rightarrow M_2$ be Riemannian submersions, $H_1$ the mean curvature vector of the submanifold $F_1$ in $E_1$ and $B_2$ the second fundamental form of the fiber submanifold $F_2$ in $E_2$. Let $f:E_1 \rightarrow E_2$ be a horizontal equivariant map and $\bar{f}$ its induced map from $M_1$ to $M_2$ with tension field $\tau(\bar{f})$. Let $f^\bot$ be the restriction of $f$ to the fiber $F_1$. Then we have the following formula
\begin{align*}
  \tau(f)={\tau} ^*(\bar{f})+B_2(f_*(e_t),f_*(e_t))-f_*(H_1)+\tau(f^\bot)
\end{align*}
 where $\{e_t\}$, $t=n_1+1,\ldots,m_1$ is a local orthonormal frame field on the fiber $F_1$ and $\tau ^*(\bar{f})$ denotes the horizontal lift of $\tau(\bar{f})$.
  \end{thm}
%%%%%%%%%%%%%%%%%%%%%%%%%%%%%%%%%%%%%%%%%%%%%%%%%%%%%%%%%%%%%%%%%%%%%%%%%%%%%%%%%%%%%%%%%%%%%%%%%%%%%%%%%%%%%%%%%%%%%%%%%%%%%%%%%%%%%%%%%%%%%%%%%%%%%%%%%%%%%%%%%%%%%%%%%%%%%%%%%%
%%%%%%%%%%%%%%%%%%%%%%%%%%%%%%%%%%%%%%%%%%%%%%%%%%%%%%%%%%%%%%%%%%%%%%%%%%%%%%%%%%%%%%%%%%%%%%%%%%%%%%%%%%%%%%%%%%%%%%%%%%%%%%%%%%%%%%%%%%%%%%%%%%%%%%%%%%%%%%%%%%%%%%%%
\subsection{H\"{o}lder spaces on manifolds}
 Let $(M,g)$ be a Riemannian manifold and let $\nabla$ be the Levi-Civita connection on $M$. Let  $V$ be a vector bundle on $M$ equipped with the Euclidean metric on its fibers. Let $\hat{\nabla}$ be a connection on $V$  preserving these metrics. Let $C^k(M)$ be the space of all continuous, bounded functions $f$ that have $k$ continuous, bounded derivatives and define the norm $\|\cdot\|_{C^k}$ on $C^k(M)$ by $\|f\|_{C^k}=\sum_{j=0}^k\sup_M|\nabla ^j f|$.

Now we define the H\"{o}lder space $C^{0,\alpha}(M)$ for $\alpha\in(0,1)$. The function $f$  on $M$ is said to be H\"{o}lder continuous with exponent $\alpha$, if
\begin{align*}
[f]_{\alpha}=\sup_{x\neq y \in M}\frac{|f(x)-f(y)|}{d(x,y)^{\alpha}}
\end{align*}
is finite. The vector space $C^{0,\alpha}(M)$ is the set of continuous, bounded functions on $M$ which are H\"{o}lder continuous with exponent $\alpha$ and the norm $C^{0,\alpha}(M)$ is $\|f\|_{C^{0,\alpha}}=\|f\|_{C^0}+[f]_{\alpha}$.

In the same way, we shall define H\"{o}lder norms on spaces  of sections $v$ of a vector bundle $V$ over $M$ equipped with Euclidean metrics in the fibers as above.
Let $\delta(g)=\inj(M,g)$ be the injectivity radius of the metric $g$ on $M$ which we suppose to be positive and set
\begin{align}\label{holder}
[v]_{\alpha}=\sup_{\stackrel{x\neq y\in M}{d(x,y)<\delta(g)}}\frac{|v(x)-v(y)|}{d(x,y)^{\alpha}}
\end{align}
We now  interpret $|v(x)-v(y)|$. When $x\neq y\in M$, and $d(x,y)\leq \delta(g)$, there is unique geodesic $\gamma$ of length $d(x,y)$ joining $x$ and $y$ in $M$. Parallel translation along $\gamma$ using $\hat{\nabla}$ identifies the fibers of $V$ over $x$ and $y$ and the metrics on the fibers. With this understanding the expression $|v(x)-v(y)|$ is well defined.

Define $C^{k,\alpha}(M)$ to be the set of $f$ in $C^k(M)$ for which  $[\nabla^k f]_{\alpha}$ defined by (\ref{holder}) exists as a section in the vector bundle $\bigotimes^k T^*M$ with its natural metric and connection. The H\"{o}lder norm on $C^{k,\alpha}(M)$ is $\|f\|_{C^{k,\alpha}}=\|f\|_{C^k}+[\nabla ^kf]_{\alpha}$.
\begin{lemme}
Let $\Omega\subset \mathbb R^n$ be a bounded domain. Suppose that $F:\Omega\rightarrow \mathbb R^q$ is bounded and H\"{o}lder continuous. Let $Q:\mathbb R^q\rightarrow \mathbb R^p$ be a quadratic function. Then $Q\circ F:\Omega\rightarrow \mathbb R^p$ is also H\"{o}lder continuous and
\begin{align*}
[Q\circ F]_{\alpha}\leq A\sup_{\Omega}\|F\|_{\mathbb R^q}[\|F\|_{\mathbb R^q}]_{\alpha},
\end{align*}
where $A$ is a constant.
\end{lemme}
In the above lemma by a quadratic function we mean
\begin{align*}
Q(y)=\sum_{i,j=1}^qQ_{ij}y_iy_j,\quad\quad Q_{ij}\in C^1({\overline{\Omega}}).
\end{align*}
We have
\begin{cor}\label{ef}
Let $f\in C^{1,\alpha}(M,N)$, then
\begin{align*}
[\Pi(f)(df,df)]_{C^{\alpha}}\leq A\cdot\|df\|_{L^{\infty}}\cdot[df]_{C^{\alpha}}.
\end{align*}
\end{cor}
\begin{proof}
Let $\{\Omega_j\}$ be an atlas of $M$,  such that $\diam(\Omega_j)\leq \inj(M)$ and set $F_j=df|_{\Omega_j}$ and $Q=\hes \pi_N(X,X)$, for an smooth vector field $X$. Then using
the previous lemma and an appropriate partition of unity we will have the result.
\end{proof}
\subsubsection*{\textbf{Schauder Estimates}} In this part, we give  a quick review on the Schauder estimate of solutions to linear elliptic partial differential equations. Suppose $(M,g)$ is  compact and $L$ is an elliptic operator, $L = a^{ij}\nabla _i\nabla_j + b_i\nabla_i + c$,
where $a$ is a symmetric and positive definite tensor, $b$ is a $C^{0,\alpha}$ vector field on $M$
and $c$ is in $C^{0,\alpha}(M)$ such that $L$ satisfies the conditions
\begin{eqnarray*}
&\|a\|_{C^{0,\alpha}}+\|b\|_{C^{0,\alpha}}+\|c\|_{C^{0,\alpha}}\leq \Lambda,\\
&\lambda\|\xi\|^2\leq a^{ij}(x)\xi_i\xi_j\leq\Lambda\|\xi\|^2,\quad\text{for all}~x\in M,~\text{and}~\xi\in \mathbb R^n.
\end{eqnarray*}
Consider the following problem,
\begin{align*}
Lu=f\quad\quad\text{in}~M,
\end{align*}
if $\partial M=\emptyset$ and
\begin{align*}
\left\{\begin{array}{ll}
Lu=f&\text{in}~M\\
u=g&\text{on}~\partial {M}.
\end{array} \right.
\end{align*}
if $\partial M\neq \emptyset$. Then we have (c.f. Gilbarg-Trudinger \cite{GT83})
\begin{thm}[Schauder Estimate]\label{schauder}
If $f\in C^{0,\alpha}(M)$ and $u\in C^2(M)$, then $u\in C^{2,\alpha}(M)$ and we have
\begin{eqnarray*}
\|u\|_{C^{1,\alpha}}&\leq& C(\|f\|_{L^{\infty}}+\|u\|_{L^{\infty}}),\\
\|u\|_{C^{2,\alpha}}&\leq& C(\|f\|_{C^{0,\alpha}}+\|u\|_{L^{\infty}}),
\end{eqnarray*}
where $C$ depends on $M$, $\lambda$, $\Lambda$.
\end{thm}
Hereafter we present an introduction to the convergence and collapsing theory. Most of the materials in this part was gathered from the work of Rong \cite{R10}.

%\email{zahra.sinaei@epfl.ch}

\subsection{Convergence}\label{convergencesubsec}
Gromov introduced the notion  of the Gromov-Hausdorff distance between  metric spaces in \cite{Gr81}), based on the notion of Hausdorff distance between  subsets $A$, $B$ in a metric space $Z$:
$$
d_H^Z(A,B)=\inf\{\epsilon>0:~B\subset T_{\epsilon}(A)~\text{and}~ A\subset T_{\epsilon}(B)\}
$$
where $T_{\epsilon}(A)=\{x\in Z:~d_Z(x,A)<\epsilon\}$ is a tubular neighborhood of a set $A$.
\begin{defin}\label{GHdis}(Gromov \cite{Gr81})
Let $X$ and $Y$ be two compact metric spaces. The Gromov-Hausdorff distance  between $X$ and $Y$ is defined as
\begin{eqnarray*}
  d_{GH}(X,Y)=\inf\left\{
d_H^Z(\phi(X),\psi(Y)):
\begin{array}{c}
\text{for all metric spaces}~Z~\text{and}~\text{isometric embeddings}\\
\phi:X \hookrightarrow Z,~\psi:Y \hookrightarrow Z
\end{array}
\right\}
\end{eqnarray*}
\end{defin}
 % where $Z$ runs over all such metric spaces and $\phi$ and $\psi$ runs over all possible isometric embedding and $d_H$ is Hausdorff distance.

 Let $\cal{MET}$ denote the set of all isometry classes of nonempty compact metric spaces. Then $({\cal{MET}}, d_{GH})$ is a complete metric space.
There is an alternative definition for Gromov-Hausdorff distance given in \cite{Gr81}:
%\subsubsection{\textbf{An alternative definition of Gromov-Hausdorff distance}}(Gromov \cite{Gr81})

\begin{defin} (Gromov \cite{Gr81}) Let $X$ and $Y$ be two elements of $\cal{MET}$.
 A map $\phi:X \rightarrow Y$ is said to be an $\epsilon$-Hausdorff approximation from X to Y, if the following two conditions are satisfied
\begin{enumerate}[i.]
\item $\epsilon$-onto: $B_{\epsilon}(\phi(X))=Y$.
\item $\epsilon$-isometry: $|d(\phi(x),\phi(y))-d(x,y)|<\epsilon$ for all $x,y\in X$.
\end{enumerate}
 The Gromov-Hausdorff distance $\hat{d}_{GH}(X,Y)$, between $X$ and $Y$ is defined to be the infimum of the positive number $\epsilon$ such that there exists $\epsilon$-Hausdorff approximation from $X$ to $Y$ and form $Y$ to $X$. 
 \end{defin} 
The distance  $\hat{d}_{GH}$ does not satisfy triangle inequality and $\hat{d}_{GH}\neq d_{GH}$ but onecan show that
\begin{align*}
 \tfrac{2}{3}d_{GH}\leq \hat{d}_{GH}\leq 2d_{GH}
\end{align*}
Because a sequence in $\cal{MET}$ converges with respect to $d_{GH}$ if and only if it converges with respect to $\hat{d}_{GH}$, we will not distinguish $\hat{d}_{GH}$ from $d_{GH}$.

For the notion of equivariant Gromov-Hausdorff convergence and equivariant measured Gromov-Hausdorff convergence, we refer the reader to Definition $1.5.2$ in \cite{R10} and Definition $3.11$ in \cite{Fuk87}. Also for the notion of Lipschitz distance see  Definition $3.1$ in \cite{Gr81}. Let $\mathcal{MM}$ denotes the class of all pairs $(X,\mu)$ of compact metric spaces $X$ equipped with 
 a Borel measure $\mu$ on it such that $\mu(X)=1$.  Fukaya in \cite{Fuk87} presented a notion of measured Gromov-Hausdorff convergence for the metric measure spaces:
\begin{defin} (Fukaya \cite{Fuk87})  Let $(X_i,\mu_i)$ be a sequence in $\cal{MM}$. We say that $(X_i,\mu_i)$ converges to an
 element $(X,\mu)$ in $\cal{MM}$ with respect to measured Gromov-Hausdorff topology if there exist Borel measurable $\epsilon$-Hausdorff
 approximations $f_i:(X_i,\mu_i)\rightarrow(X,\mu)$ such that ${f_i}_*(\mu_i)$ converges to $\mu$ in the $\text{weak}^*$ topology.
\end{defin}
When $M$ is a Riemannian manifold with finite volume, we  let $\mu_M=\tfrac{\dvol_ M}{\vol(M)}$, where $\dvol_M$ denotes the volume element of $M$ and regard $(M,\mu_M)$ as an element in $\cal{MM}$.
%%%%%%%%%%%%%%%%%%%%%%%%%%%%%%%%%%%%%%%%%%%5555

In \cite{GP91}, Grove and Petersen introduced the notion of convergence of maps. 
\begin{defin}(Grove-Petersen \cite{GP91}) \label{convergence-maps} Let $(X_i,p_i)$, $(X,p)$, $(Y_i,q_i)$ and $(Y,q)$ be  pointed metric spaces
 such that $(X_i,p_i)$ converges to $(X,p)$ in the pointed Gromov-Hausdorff topology (resp. $(Y_i,q_i)$ converges to $(Y,q)$).
 We say that a sequence of maps  $f_i:(X_i,p_i)\rightarrow(Y_i,q_i)$ converges to a map $f:(X,p)\rightarrow(Y,q)$ if  
there exists a subsequence $X_{i_k}$ such that if $x_{i_k}\in X_{i_k}$ and $x_{i_k}$ converges to $x$ (in $\coprod X_{i_k}\coprod X$ with 
the admissible metric), then  $f_{i_k}(x_{i_k})$ converges to $f(x)$. 
\end{defin}
A family of maps $f_i:(X_i,d_{X_i},p_i)\rightarrow(Y_i,d_{Y_i},q_i)$ is called equicontinuous if for any $\epsilon>0$ there is $\delta>0$ such that $d_{X_i}(x_i,y_i)<\delta$ implies $d_{Y_i}(f_i(x_i),f_i(y_i))<\epsilon$ for all $x_i,y_i$
 in $X_i$ and for all $i$. We have

\begin{lemme}\label{Rong} (Grove-Petersen \cite{GP91})  Let $(X_i,p_i)$, $(X,p)$, $(Y_i,q_i)$ and $(Y,q)$ be  pointed metric spaces
 such that $(X_i,p_i)$ converges to $(X,p)$ in the pointed Gromov-Hausdorff topology (resp. $(Y_i,q_i)$ converges to $(Y,q)$).  Let $f_i:(X_i,p_i)\rightarrow(Y_i,q_i)$ be a sequence of maps. Then
\begin{enumerate}[i.]
 \item If $f_i$s are equicontinuous, then there is a uniformly continuous map $f$ and a convergent subsequence $X_{i_k}$ such that  $f_i$ converges to $f$.
\item  If $f_i$s are isometries then the limit map $f:(X,p)\rightarrow(Y,q)$ is also an isometry.
 \end{enumerate}
\end{lemme}
%For the proof of the above lemma see \cite{R10} Lemma $1.6.12$.

\subsection{Convergence Theorems, Non-Collapsing}
This subsection is devoted to the theory of convergence of manifolds in the non-collapsing case.
 A sequence of $n$-manifolds $M_i$ converging to a metric space $X$ is called non-collapsing
if $\vol(M_i)\geq v>0$, and  collapsing otherwise.  For a non-collapsing sequence of manifolds with bounded sectional curvature
 there is a uniform lower bound on the injectivity radius of $M_i$, and thus $M_i$s are diffeomorphic for large $i$. 
This result is due to Cheeger-Gromov (Cheeger \cite{Ch70}, Peters \cite{P84}, Greene-Wu \cite{GW89}) and is formulated as follows.
\begin{thm}\label{CGP}
Let $(M_i,g_i)$ be a sequence of closed Riemannian $n$-manifolds such that $|\sect_{g_i}|\leq 1$  and $\vol(M_i)\geq v>0$, and $M_i$ converges to a metric space $X$.
 Then $X$ is homeomorphic to a manifold $M$ such that for large $i$,  and
there are diffeomorphisms $\phi_i:M\rightarrow M_i$ such that the pullback metric converges to a $C^{1,\alpha}$-metric  $g$ on $M$ in the $C^{1,\alpha}$-topology.
\end{thm}
The following smoothing result concerns the uniform approximation of Riemannian manifolds by smooth ones.
\begin{thm}[Bemelmans-Oo-Ruh \cite{B84}]\label{bemelman}
Let $(M,g)$ be a compact Riemannian $n$-manifold with $|\sect_g|<1$. For any $\epsilon>0$, there is a smooth metric $g_{\epsilon}$ on $M$ such that
\begin{align*}\label{Bemelman}
|g_{\epsilon}-g|_{C^1}<\epsilon,\quad|\sect_{g_{\epsilon}}|\leq 1, \quad|\nabla ^k\R_{g_{\epsilon}}|\leq C(n,k)\cdot {\epsilon}^k.
\end{align*}
In particular
\begin{eqnarray*}
%\left\{
%\begin{array}{c}
&\e^{-\epsilon}\inj(M,g)\leq \inj(M,g_{\epsilon})\leq  \e^{\epsilon}\inj(M,g),\\
&\e^{-\epsilon}\diam(M,g)\leq \diam(M,g_{\epsilon})\leq  \e^{\epsilon}\diam(M,g),\\
&\e^{-\epsilon}\vol(M,g)\leq \vol(M,g_{\epsilon})\leq  \e^{\epsilon}\vol(M,g).
%\end{array}
%\right.
\end{eqnarray*}
\end{thm}
%The following lemma shows that smooth elements are dense in the set of compact Riemannian $n$-manifolds with bounded curvature.

\subsection{Convergence Theorems-Collapsing}
This subsection is devoted to the theory of convergence of manifolds in the collapsing case. We state some of the main results in this context.
%\begin{thm}[Fibration theorem, Fukaya  \cite{F87}]\label{fbrbndl}
%Let $M^n$ and $N^m$ be compact manifolds satisfying
%\begin{eqnarray*}
%\sect_{M^n}\geq -1,~~~|\sect_{N^m}|\leq 1~(m\geq2),~~~\inj(N^m)\geq i_0>0
%\end{eqnarray*}
%There exists a constant $\epsilon(n,i_0)$ such that if $d_{GH}(M^n,N^m)<\epsilon\leq\epsilon(n,i_0)$, then there is a $C^1$-fibration map $f:M^n\rightarrow N^m$ with connected fibre such that

%i) The diameter of any $f$-fibres is at most $c\cdot\epsilon$, where $c=c(n,\epsilon)$ is such that $c\rightarrow 1$ as $\epsilon\rightarrow 0$.\\
%ii) $f$ is an almost Riemannian submersion , that is for any vector $\xi\in TM$ orthogonal to a fibre,
%\begin{align*}
%\e^{-\tau(\epsilon)}\leq \frac{|df(\xi)|}{|\xi|}\leq \e^{\tau(\epsilon)},
%\end{align*}

%where $\tau(\epsilon)\rightarrow0$ as $\epsilon\rightarrow0$.\\
%iii) If in addition, $\sect_{M^n}\leq 1$, then $f$ is smooth and the second fundamental form of any fiber satisfies $|II_{f^{-1}(\bar{x})}|\leq c(n)$.
%\end{thm}
%For the proof ot the theorem above see Theorem 2.1.1 in \cite{R10}. We have the following complement of the previous theorem.
%Theorem $2.6$ in \cite{CGF92} and Theorem 0-1 in \cite{F89}
\begin{thm}[Fibration theorem, Fukaya \cite{F89}, Cheeger-Fukaya-Gromov \cite{CGF92}]\label{fbrbndl}
Let $M^n$ and $N^m$ be compact Riemannian manifolds satisfying
\begin{align*}
\sect_{M^n}\geq -1,\quad |\sect_{N^m}|\leq 1~(m\geq 2),\quad  \inj(N^m)\geq i_0 > 0.
\end{align*}
Assume  $M^n$ and $N^m$ admit isometric compact Lie group $G$-actions. There exists a constant $\epsilon(n,i_0)>0$ such that
 if $d_{eqGH}((M^n,G),(N^m,G))<\epsilon\leq\epsilon(n,i_0)$,  then there is a $C^1$-fibration $G$-invariant map, $f:(M^n,G)\rightarrow (N^m,G)$ with connected fibers such that 
\begin{enumerate}[i.]
\item The diameter of any $f$-fibers is at most $c_1\cdot\epsilon$, where $c_1=c_1(n,\epsilon)$ is such that $c_1\rightarrow 1$ as $\epsilon\rightarrow 0$.
\item  $f$ is an almost Riemannian submersion, that is for any vector $\xi\in TM$ orthogonal to a fiber,
\begin{align*}
\e^{-\tau(\epsilon)}\leq \frac{|df(\xi)|}{|\xi|}\leq \e^{\tau(\epsilon)},
\end{align*}
where $\tau(\epsilon)\rightarrow0$ as $\epsilon\rightarrow0$.
\item  If in addition, $\sect_{M^n}\leq 1$ then $f$ is smooth and the second fundamental form of any fiber satisfies $|II_{f^{-1}(\bar{x})}|\leq c_2(n)$, for $\bar{x}$ in $N^m$.
\item  The fibers are diffeomorphic to an infranilmanifold $\Gamma\backslash \mathring{N}$, where $\mathring{N}$ is a simply connected nilpotent group, $\Gamma\subset \mathring{N}\ltimes \aut(\mathring{N})$, such that $[\Gamma,\mathring{N}\cap\Gamma]\leq \omega(n)$.
\end{enumerate}
%v) There are canonical flat connections on fibres that vary continuously and the $G$-action preserves flat connections.\\
%vi) The structure group of the fibration is contained in $\frac{\cent(N)}{\cent(N)\cap\Gamma}\ltimes\aut(N\cap\Gamma)$.
\end{thm}
An easily accessible proof of this theorem  can be founded in \cite{R10}  Theorems 2.1.1 and 5.7.1.\\

A {\it{pure nilpotent Killing structure}} on $M^n$  is a $G$-equivarient fibration $N_0\rightarrow M^n\rightarrow N^m$, with fiber ${N_0}$ a nilpotent manifold
 (equipped with a flat connection) on which parallel fields are Killing fields and the $G$-action preserves affine fibrations.
 The underlying $G$-invariant affine bundle structure is called a pure ${N_0}$-structure
 and a metric for which the ${N_0}$-structure becomes a nilpotent Killing structure is called invariant.

Let $M^n$ and $N^m$ be as in Theorem \ref{fbrbndl}. Suppose $M^n$ and $N^m$ satisfy the following: for some  sequence $A=\{A_k\}$ of real non-negative numbers, for the Riemannian curvature tensor on $M$ and $N$  
we have
\begin{eqnarray}\label{regular}
|\nabla^k \R|\leq A_k.
\end{eqnarray}
  We can construct an invariant metric (invariant under the left action of ${N_0}$) such that
%and the metric for which the $N$-structure becomes a nilpotent Killing structure is called invariant.
\begin{eqnarray}\label{cgf}
|\nabla^k(\langle\quad,\quad\rangle-(\quad,\quad))|\leq c(n,A)\cdot\epsilon\cdot \inj(N)^{-(k+1)},
\end{eqnarray}
where $\langle\quad,\quad\rangle$ denotes the original metric, $(\quad,\quad)$ the invariant one, and
$c(n,A)$ is a generic constant depending on finitely many $A_k$ and $n$. For the construction of invariant metric which satisfies inequality (\ref{cgf}) see Proposition 4.9 in \cite{CGF92} and the explanation therein.
Given such a metric we have a pure nilpotent killing structure.

%By the Fibration Theorem, each fiber is an almost flat manifold and so is diffeomorphic to an infranilmanifold $\Gamma\backslash N$. First any Lie group has a canonical flat connection $\nabla^{\can}$, for which left invariant fields are parallel. The $N$-action on $\Gamma\backslash N$ from the right generate the left invariant fields on $\Gamma\backslash N$ and thus defines the canonical flat connection. By Malcev's rigidity theorem, any affine structure on $\Gamma\backslash N$ are affine equivalent to $(\Gamma\backslash N,\nabla^{\can})$. If the flat connection on the fibre can be chosen smoothly, then the structure group of the fibration reduces to a subgroup of the affine transformation on $\Gamma\backslash N$. The  flat connection on  $\Gamma\backslash N$ may depend on the choice of base point on the fiber. By averaging among all  flat connections from the various choices of base points, a continuous family of flat connections on each fibers are constructed. In this way we can construct a canonical
%invariant metric for the left action of $N$.
When a sequence of  Riemannian $n$-manifolds with bounded curvature collapses, the limit space can be a singular space. We have
\begin{thm}[Singular fibration theorem, Fukaya \cite{F88}]\label{singular}
Let $(M_i,g_i)$ be a sequence of closed Riemannian $n$-manifolds with $|\sect_{g_i}|\leq 1$ and $\diam(M_i)\leq D$ which converges to the closed metric space $(X,d)$ in $\mathcal{MET}$. Then
\begin{enumerate}[i.]
\item The frame bundles equipped with canonical metrics converge, $(F(M_i),O(n))\rightarrow (Y,O(n))$, where $Y$ is  a manifold. 
\item There is an $O(n)$-invariant fibration  $\tilde{f}_i:F(M_i)\rightarrow Y$ satisfying the conditions in Theorem \ref{fbrbndl} which becomes for $\epsilon>0$, a nilpotent Killing structure with respect to an $\epsilon$  $C^1$-closed metric (with respect to $C^1$-topology). Moreover each fiber on $M_i$ has positive dimension.
\item  For any $\bar{x}\in X$, a fiber $f_i^{-1}(\bar{x})$ is singular if and only if $p^{-1}(\bar{x})$ is a singular $O(n)$-orbit in $Y$.
\end{enumerate}
\end{thm}

For the proof see Theorem $4.1.3$ in \cite{R10}. In the above theorem, the fibration map $\tilde{f}_i$ descends to a (singular) fibration map $f_i:M_i\rightarrow X=Y\slash O(n)$ such that the following diagram commutes\\
\begin{center}
\begin{tikzcd}
F(M_i) \arrow{r}{{\tilde{f}_i}} \arrow{d}{p_i}
&Y \arrow{d}{p}\\
M_i \arrow{r}{f_i} &X
\end{tikzcd}
\end{center}
In the following remark we collect the main points that we need from the  above theorems and explain the classification in the proof of Theorem \ref{main}.
 \begin{remark}\normalfont\label{explain}
 When a sequence of Riemannian manifolds $M_i$ converges in $\mathcal{M}(n,D)$ to a metric space $X$, 
the frame bundles over $M_i$ equipped with the canonical metrics $\tilde{g_i}$ converge to a manifold $Y$ and $\tilde{f_i}:(F(M_i),\tilde{g_i},O(n))\rightarrow (Y,O(n))$ is an $O(n)$-invariant fibration map.

 To see this let $\tilde{g_i}_{\epsilon}$ be the smooth metric on $F(M_i)$ as in Theorem \ref{bemelman}. Then $(F(M_i),\tilde{g_i}_{\epsilon})$ converges to a smooth Riemannian manifold $(Y_{\epsilon},g_\epsilon)$.
 For a small fixed $\epsilon_0$ and  $\epsilon<\epsilon_0$, the sectional curvature on $(F(M_i), \tilde{g_i}_{\epsilon})$ is uniformly bounded and
 we can apply  Theorem  \ref{fbrbndl} to conclude that there exists an $O(n)$-invariant smooth fibration map $\tilde{f_i}_{\epsilon}$.
 By  continuity $(F(M_i),\tilde{g_i}_{\epsilon})$ is conjugate to $(F(M_i),\tilde{g_i}_{\epsilon_0})$
 (by being conjugate we mean there exists  $C^{1,\alpha}$-diffeomorphism as in Theorem \ref{CGP}). 
This implies that the convergence of  $Y_{\epsilon}$ to $Y$ is the same as the convergence of a sequence of metrics on $Y_{\epsilon_0}$, and therefor $(Y,O(n))$ is conjugate  to $(Y_{\epsilon_0},O(n))$
 \begin{align*}
 (F(M_i),O(n))\simeq (F(M_i), \tilde{g_i}_{\epsilon_0},O(n))\stackrel{\tilde{f_i}_{\epsilon_0}}{\rightarrow}(Y_{\epsilon_0},O(n))\simeq(Y,O(n)),
 \end{align*}
 and it induces a fibration map $(F(M_i),\tilde{g_i},O(n))\stackrel{\tilde{f_i}}{\rightarrow}(Y,O(n))$ . For more explanations see the proof of Theorem  $4.1.3$ in \cite{R10}.
 
Furthermore,  there exists a $C^1$-close invariant Riemannian metric $\mathring{g_i}_{\epsilon}$ such that  $(F(M_i),\mathring{g_i}_{\epsilon},O(n))$ is a pure nilpotent Killing structure and the fibration map $\tilde{f_i}_{\epsilon}$ is a Riemannian submersion considering the induced Riemannian metric on $Y_{\epsilon}$ by this map.
 \end{remark}
\subsection{Density function}\label{densityfunction}
Let $\mathcal{DM}(n,D)$ denote the closure of $\mathcal{M}(n,D)$ in $\mathcal{MM}$ with respect to the measured Gromov-Hausdorff topology. Then $\mathcal{DM}(n,D)$ is compact with respect to the measured Gromov-Hausdorff topology. Let $(M_i,g_i,\tfrac{\dvol_{M_i}}{\vol(M_i)})\in \mathcal{M}(n,D)$ be a sequence of manifolds which converges to a manifold $(M,g,\mu)$.
Suppose  $\psi_i:M_i\rightarrow M$ is the fibration map  as in Theorem \ref{CGP}. For $x\in M$ we define
\begin{align*}
 \Phi_i=\tfrac{\vol(\psi_i^{-1}(x))}{\vol(M_i)},
\end{align*} 
then there exists $\Phi$ such that $\Phi=\lim_{i\rightarrow \infty}\Phi_i$ and $\mu$ is absolutely continuous with respect to $\dvol_M$, 
$\mu=\Phi\dvol_M$ (see \S3 in \cite{Fuk87}). For the general case when $(X,\mu)\in \mathcal{DM}(n,D)$, we first recall  a remark on quotient spaces. Below $S(B)$ denotes the singular part of $B$.
\begin{remark}[Besse \cite{BE87}]\label{que}
\normalfont  Let $(M,g)$ be a Riemannian manifold and $G$  a closed subgroup of isometries of $M$. Assume that the projection $p:M\rightarrow M\slash G$ is a smooth submersion. Then there exists a unique Riemannian metric $\check{g}$ on $B=M\slash G$ such that $p$ is a Riemannian submersion  (see Subsection 9.12 in \cite{BE87}).\\
  We recall that using the general theory of slices for the action of a group of isometries on a Riemannian manifold,  one can show that there always exists an open dense submanifold $U$ of $M$ (the union of the principle orbits), such that the restriction  $p|_U:U\rightarrow U\slash G$ is a smooth submersion.
  
  Considering now $M\slash G$ as a Riemannian polyhedron and $\mu_g$ as its Riemannian volume element, the restriction of $\mu_g$ on $U\slash G$ is equal to $\dvol_{U\slash G}=\dvol _{B-S(B)}$.
  \end{remark}%[\cite{F89}, \S3]
Now suppose $M_i$ in $\mathcal{M}(n,D)$ converges to a metric space $X$. We may assume that $FM_i$ with the induced $O(n)$-invariant metric $\tilde{g_i}$ converges to $(Y,g,\Phi_Y\cdot \dvol_Y)$ with respect to the $O(n)$-measured Gromov-Hausdorff topology and $g$, $\Phi_Y$ are  $C^{1,\alpha}$-regular.
 Moreover, since $p_i:F(M_i)\rightarrow M_i$ is a Riemannian submersion with totally geodesic fibers, and since the fibers are
isometric to each other, it follows that $(FM_i, \dvol_{FM_i})/O(n)=(M_i,\dvol_{M_i})$. Hence
 by equivariant Gromov-Hausdorff convergence $M_i$ converges to $(X,\nu)=(Y,\Phi_Y\dvol_Y)\slash O(n)$ (see Theorem 0.6 in \cite{Fuk87}), and by Remark \ref{que}
\begin{align*}
\nu(S(X))=0
\end{align*}
For all $x$ in $X$ we let
\begin{align*}
\Phi_X(x)=\int_{y\in p^{-1}(x)}\Phi_Y(y)~\dvol_{p^{-1}(x)},
\end{align*}
where $p:Y\rightarrow X$ is the natural projection. For each open set $U$
\begin{align*}
\nu(U)=\int_{U}\Phi_X(x)~\dvol_{X-S(X)}.
\end{align*}
%If  we denote the set of all points for whi%\begin{figure}[h]
%$$\xymatrix{
%F(M_i) \ar[rr]^-{\tilde{f}_i} \ar@<2pt>[dd]^-{p_i}   & & Y  \ar@<2pt>[dd]^-{p} \\
%& &  \\
%M_i \ar[rr]^-{f_i}    &  & X
%}$$\\
%\end{figure}ch the isotropy group is not finite by $\tilde{S}(X)$, then
%\begin{align*}
%\tilde{S}(X)=\{x\in X|\Phi_X(x)=0\}.
%%\end{align*}
%
%In the following ${\cal{M}}(n,D)$ denotes the set of all compact Riemannian manifold  $(M,g)$ such that $\dime(M)=n$, $\diam(M)<D$ and
%sectional curvature $|\sect_g|\leq 1$, and  ${\cal{M}}(n,D,v)$ the set of Riemannian manifolds in ${\cal{M}}(n,D)$ with volume $\geq v$.
%----------------------------------------------------------------------------------------
\section{Proof of the Convergence Theorem}
In this section we are going to prove Theorem \ref{main}. In the following ${\cal{M}}(n,D)$ denotes the set of all compact Riemannian manifolds  $(M,g)$ such that $\dime(M)=n$, $\diam(M)<D$ and
the sectional curvature satisfies $|\sect_g|\leq 1$, and  ${\cal{M}}(n,D,v)$ the set of Riemannian manifolds in ${\cal{M}}(n,D)$ with volume $\geq v$.

We split the proof in three cases:\\\\
\textbf{Case I: Non-collapsing.} $(M_i,g_i)$ converge to $(M,g)$ in ${\cal{M}}(n,D,v)$.
We first consider  the situation where $M_i=M$ and $g_i$  converges to a metric $g$ in ${\cal{M}}(n,D,v)$. Then we study the problem in the general case using  Theorem \ref{CGP}.\\\\
\textbf{Case II: Collapsing to a manifold.} $(M_i,g_i)$ converge to  $(M,g)$ in ${\cal{M}}(n,D)$ with $g$ a $C^{1,\alpha}$-metric.
 We first consider the situation when $(M_i,g_i)$ satisfies an additional regularity assumption (see Assumption \ref{assumption} below). Then we discuss the general case using the fact that there is always a sequence of  metrics $g_i(\epsilon)$ on $M_i$, $C^1$-close to the the metric $g_i$ which satisfies Assumption \ref{assumption} as  explained in Remark \ref{explain}.\\\\
 \textbf{Case III: Collapsing to a singular space.} $(M_i,g_i)$ converge to a metric space $(X,d)$ in ${\cal{M}}(n,D)$.
 When a sequence of manifolds $(M_i,g_i)$ converges in $\mathcal{M}(n,D)$ to a metric space $X$, the frame bundles over $M_i$ converge to a Riemannian manifold $Y$, with a $C^{1,\alpha}$-metric and we have $X=Y/O(n)$. The harmonic maps over $M_i$, induce harmonic maps over $F(M_i)$ and this case reduces to the study of harmonic maps on quotient spaces. \\
 %When a sequence of manifolds $M_i$ converges in $\mathcal{M}(n,D)$ to a manifold, they form a fiber bundle over the limit manifold. As a main step in the proof, we show that the $f_i$s are almost constant on the fibers. For this, we were inspired by Theorem $0.4$ in Fukaya  \cite{Fuk87}.\\

 Hereafter we fix an isometric embedding $I: N \to  R^q$  and we often denote the composition  $I \circ f$  simply by $f$,
unless we need to explicitly distinguish these two maps.

\subsection{Case I: Non-collapsing.}\label{casei}
In this subsection we prove
 \begin{prop}\label{CaseI}
Let $(M_i,g_i)$ be a sequence of Riemannian manifolds in ${\mathcal{M}}(n,D,v)$ which converges to a Riemannian manifold $(M,g)$ in the Gromov-Hausdorff topology. Suppose $(N,h)$ is a compact Riemannian manifold. Let $f_i:(M_i,g_i)\rightarrow (N,h)$ be a sequence of smooth harmonic maps such that $\|e_{g_i}(f_i)\|_{L^{\infty}}<C$, where $C$ is a constant independent of $i$. Then $f_i$ has a subsequence  which converges to a map  $f:(M,g)\rightarrow (N,h)$ and this map is a smooth harmonic map.
\end{prop}
To go through the proof in this case, we first consider the situation when a sequence of metrics $g_i$ on a manifold $M$ converges to a Riemannian metric $g$.
\begin{lemme}\label{one}
Let $g_i$ be a sequence of Riemannian metrics on a smooth manifold $M$ and  suppose $(M,g_i)$ converge to $(M,g)$ in ${\cal{M}}(n,D,v)$. Suppose $f_i:(M,g_i)\rightarrow N$ is a sequence of smooth harmonic maps such that
\begin{align*}
\|e_{g_i}(f_i)\|_{L^{\infty}}<C,
\end{align*}
where $C$ is a constant independent of $i$. Then there exists a subsequence of $f_i$ which converges to some $f$ in the $C^k$-topology for any $k\geq0$ and $f$ is also harmonic.
\end{lemme}
\begin{proof}
By Theorem \ref{CGP}, the metric $g_i$ converges to $g$ in ${\cal{M}}(n,D,v)$ in the $C^{1,\alpha}$-topology.  Using Schauder estimates, $f_i$s have bounded norm in $C^k(M)$ for every $k\geq 0$ and hence converge  to a map $f\in C^k(M)$. We have
\begin{align*}
\lim _{i\rightarrow \infty}\Delta_{g_i}f_i=\Delta_{g} f
\end{align*}
and
\begin{align*}
\lim_{i\rightarrow \infty}\Pi(f_i)(df_i,df_i)=\Pi(f)(df,df)
\end{align*}
 The above limits lead to harmonicity of $f$.
\end{proof}
Using the above lemma we can prove Proposition \ref{CaseI}.\\
\begin{proof}[Proof of Proposition \ref{CaseI}]  Since $M_i$ converges to $M$ in ${\cal{M}}(n,D,v)$, by Theorem \ref{CGP} there is a diffeomorphism $\phi_{i}:M_{i}\rightarrow M$, such that the pushforward  $\bar{g}_i={\phi_{i}}_{*}(g_{i})$ of the metrics $g_{i}$ on $M_{i}$ converges to a $C^{1,\beta}$-metric $g$. Since the map $\phi_i:(M_i,g_i)\rightarrow (M,\bar{g}_i)$ is an isometry
\begin{eqnarray}\label{energy}
 e_{g_i}(f_i)=e_{\bar{g}_i}(\bar{f}_i)
 \end{eqnarray}
 where $\bar{f_i}$ is the map $f_i \circ \phi^{-1}_{i}$. $f_i$ is harmonic and so $\bar{f_i}$. Therefore all the assumptions of Lemma \ref{one} are satisfied here and the proof of Theorem \ref{main} in this case is complete.
\end{proof}
%%%%%%%%%%%%%%%%%%%%%%%%%%%%%%%%%%%%%%%%%%%%%%%%%%%%%%%%%%%%%%%%%%%%%%%%%%%%%%%%%%%%%%%

In Lemma \ref{one} if we replace the assumption of uniform boundedness of the energy density $\|e_{g_i}(f_i)\|_{L^{\infty}}<C$ with the assumption uniform bound on the energy $E_{g_i}(f_i)<C$, then the limiting map is not necessarily harmonic (see Theorem \ref{reg-sch} and Remark \ref{lin}).

\begin{prop}\label{bndenrg}
Let $(M_i,g_i)$ be a sequence of manifolds in ${\mathcal{M}}(n,D,v)$ which converges to a Riemannian manifold $(M,g)$ in the measured Gromov-Hausdorff topology. Suppose $(N,h)$ is a compact Riemannian manifold which does not carry any harmonic 2-sphere $S^2$. Let $f_i:(M_i,g_i)\rightarrow (N,h)$ be a sequence of harmonic maps such that $E_{g_i}(f_i)<C$ where $C$ is a constant independent of $i$. Then $f_i$ has a subsequence  which converges to a map  $f:(M,g)\rightarrow (N,h)$, and this map is a weakly  harmonic map.
\end{prop}
\begin{proof}
With the same argument as in the proof of Proposition \ref{CaseI} we consider $f_i$ and $g_i$ to be on the manifold $M$. When we have a sequence of Riemannian manifolds $(M, g_i)$ which converges in $\mathcal{M}(n,d,v)$,   the injectivity radius is bounded from below and  $\dvol_{g_i}$ converges to $\dvol_g$ weakly. Therefore if $E_{g_i}(f_i)<C$, $C$ independent of $i$, then $E_g(f_i)$ is uniformly bounded. Adapting the proof of  Remark \ref{lin} for our case,  $f_i$ converges strongly in $\mathcal{H}^1$ to a map $f$.
Also $\hes(\pi_N)$  restricted to a neighborhood of $N$ is Lipschitz and  $\hes(\pi_N)\circ f_i$ converges to $\hes(\pi_N)\circ f$ in $\mathcal{H}^1$-norm  (see Lemma $6.4$ in Taylor's book \cite{T00}) and so therefore  $\Pi(f_i)(df_i,df_i)$ converges  weakly to $\Pi(f)(df,df)$. We have the same for $\Delta f_i$ and so $f$ is a weakly harmonic map.
\end{proof}
%\textcolor{blue}{One can improve the above proposition and show that $f$  is a smooth harmonic map.  The map $f_i$ is harmonic and so energy minimizing map. Therefore  $\bar{f}_i$ is  energy minimizing $(M,\bar{g}_i)$.  Furthermore $\bar{f}_i$ is an energy minimizing map on $M$ with respect to the metric $g$.  This follows from the definition of energy minimizing maps.  By Theorem IV in \cite{SchUhl}, we have $f$  is smooth. Since in this paper we don't discuss energy minimizing maps we leave this generalization for somewhere else.}
Under the assumptions of the above theorem one can show more and prove $f$ is stationary harmonic. Under stronger assumptions on $N$ or on the image of $f$, we can show that the limit map $f$ is strongly harmonic. These results are direct consequences of some of the theorems in \cite{S83}.
\begin{prop}\label{bndenrg1}
Let $(M_i,g_i)$ and $f_i$ be as in Proposition \ref{bndenrg}. Then the map $f$ is smooth harmonic, provided that $N$ is a compact Riemannian manifold and we have one of the following conditions:
\begin{enumerate}[i.]
\item $(N,h)$ is a non-positively curved Riemannian manifold.
\item There is no strictly convex bounded function on $f(M)$.
\end{enumerate}
\end{prop}
\begin{proof}
\begin{enumerate}[i.]
\item See Proposition 2.1 in \cite{S83}.
\item  See Corollary 2.4 in \cite{S83}.
\end{enumerate}
\end{proof}

\subsection{Case II: Collapsing to a manifold.}\label{caseii}
In this subsection we prove
 \begin{prop}\label{CaseII}
Let $(M_i,g_i)$ be a sequence of Riemannian manifolds in ${\mathcal{M}}(n,D)$ which converges to a Riemannian manifold $(M,g,\Phi \dvol_M)$ in the measured Gromov-Hausdorff topology with $C^{1,\alpha}$-pair  $(g,\Phi)$. Suppose $(N,h)$ is a compact Riemannian manifold. Let $f_i:(M_i,g_i)\rightarrow (N,h)$ be a sequence of smooth harmonic maps such that $\|e_{g_i}(f_i)\|_{L^{\infty}}<C$, where $C$ is a constant independent of $i$. Then $f_i$ has a subsequence  which converges to a map  $f:(M,g,\Phi\dvol_g)\rightarrow (N,h)$, and this map is a weakly harmonic map.
\end{prop}
Before we prove the proposition in general, we will prove the following proposition which has an additional regularity assumption. Then at the end of this subsection, we will apply this proposition to prove case II. Consider the following assumption,
\begin{asum}\label{assumption}
Let the Riemannian metric $g_i$ be regular on $M_i$, i.e. there exists a sequence $C=\{C_k\}$ of positive number $C_k$ independent of $i$, such that
\begin{eqnarray}\label{curv}
 |{\nabla_{g_i}^k}\R_{g_i}|<C_k.
\end{eqnarray}
Suppose also that  the Riemannian metric $g_i$ is an invariant metric with respect to the nil-structure.
\end{asum}
We have
\begin{prop}\label{jadid}
Let $(M_i,g_i)$ be a convergent sequence of Riemannian manifolds in ${\mathcal{M}}(n,D)$ (with respect to the measured Gromov-Hausdorff topology) such that $g_i$ satisfies the Assumption \ref{assumption}. Let $(M,g,\Phi)$ be the limit manifold. Suppose $(N,h)$ is a compact Riemannian manifold. Let $f_i:(M_i,g_i)\rightarrow (N,h)$ be a sequence of smooth harmonic maps such that $\|e_{g_i}(f_i)\|_{L^{\infty}}<C$, where $C$ is a constant independent of $i$. Then $f_i$ has a subsequence  which converges to a map  $f:(M,g,\Phi\dvol_M)\rightarrow (N,h)$ and this map is a smooth harmonic map.
\end{prop}
Before we prove the Proposition \ref{jadid}, we first recall a few remarks from \cite{F88,F89}. Then we prove Lemma \ref{main1} which is the main element in the proof of Proposition \ref{jadid}.
\begin{remark}\normalfont\label{goosaale}
In \cite{F89} Fukaya proves that with the extra regularity assumption  (\ref{curv}) on $g_i$, $(M_i,g_i,\tfrac{\dvol_{M_i}}{\vol(M_i)})$ converges to a smooth Riemannian manifold, with the smooth pair $(g,\Phi)$. See Lemma 2.1 in \cite{F89}. By Theorem \ref{fbrbndl}, we know that for $i$ large enough,  there is a fibration map $\psi_i:M_i\rightarrow M$.  Since $g_i$ is an invariant metric, there exist  metrics $g_i^M$ on $M$ such that the maps $\psi_i:(M_i,g_i)\rightarrow (M,g_i^M)$   are Riemannian submersions and  $g_i^M$ converges to $g$ as in Theorem \ref{CGP}.
\end{remark}

\begin{remark}[Fukaya \cite{F88,F89}]\label{fukk}
\normalfont Take an arbitrary point $p_0$ in $M$ and choose $p_i\in\psi_i^{-1}(p_0)$. By  $|\sect_{g_i}|\leq 1$, at point $p_i$ on $M_i$ the conjugate radius\footnote{The conjugate domain at a point $p$ in a Riemannian manifold $M$  is the largest star shaped domain in which $d\exp_p$ is non-singular and the conjugate radius is the radius of the largest ball in the conjugate domain at $p$.} is greater than some constant name it $\rho$. We name the pullback of the Riemannian metric $g_i$ by the exponential map, $\exp_{p_i}$ at $p_i$, $\tilde{g}_i$. Therefore the injectivity radius at $0$ is at least the conjugate radius at $p_i$ (see Corollary 2.2.3 in \cite{R10}).
%(The exponential map at conjectivity domain is non singular and so it is local diffeomorphism and so we can pull back the metric and with the new
%For $\epsilon>0$,
% \begin{eqnarray*}
% \Gamma_i(\epsilon)=\{\gamma, \gamma ~\text{is a geodesic loop at}~$x_i$ ~\text{of length}~ <\epsilon\}\slash \sim,
% \end{eqnarray*}
%where $\gamma\sim\gamma'$ if and only if $\gamma$ is homotopy equivalent to $\gamma'$ via loops of length at most $2\epsilon$. We define the multiplication in $\Gamma_i(\epsilon)$ via the loop lifting. Let $\Gamma_i$ denote the pseudogroup generated by these elements.Then $\Gamma_i$ partially and freely acts on $TB_{2\epsilon}(0)$ by isometries such that $TB_{2\epsilon}(0)\slash \Gamma_i=B_{\epsilon}(x_i)$. (see chapter 2, page 26-27 Fukaya's lecture note, \cite{R10} page 52, and chapter 3 in \cite{F88}.) $\Gamma_i$ converges to a Lie Group $\Gamma$. For the following we choose $\epsilon\leq\rho/2$

Consider the ball $B=B(0,\rho)$ in $T_{p_i}M_i$ with the metric $\tilde{g_i}$. By virtue of the regularity assumption on $g_i$, $\tilde{g_i}$ will converge to some $g_0$ in the $C^{\infty}$-topology. There are local groups $G_i$ converging to a Lie group germ $G$ such that
\begin{enumerate}[1.]
 \item $G_i$ act by isometries  on  the pointed metric spaces $((B,\tilde{g_i}),0)$.

 \item $((B,\tilde{g_i}),0)/{G_i}$ is isometric to a neighborhood of $p_i$ in $M_i$.

 \item $G$ acts by isometries on the pointed metric space $((B,g_0),0)$.

 \item  $((B,g_0),0)/{G}$ is isometric to a neighborhood of $p_0$ in $M$ and the action of G is free.
\end{enumerate}
It follows that there is a neighborhood $U$ of $p_0$ in $M$ and a $C^{\infty}$ map $s:U\rightarrow B$ such that
\begin{enumerate}[i.]
\item $s(p_0)=0$.

\item  $P\circ s=Id$, where $P$ denotes the composition of the projection map and the above mentioned isometry in $4$.

\item $d_{(B,g_0)}(s(q),0)=d_M(q,p_0)$ holds for $q\in M$.
  % actually $s:B\s0lash G\rightarrow B$, which is isometry along the radials and we don't know more.
\end{enumerate}
Therefore there is some  constant, which we again name $\rho$, independent of $i$ such that, $M=\bigcup_{j=1}^{m}B_{\tfrac{\rho}{2}}(x_j,M)$ and $B_{\tfrac{\rho}{2}}(x_j,M)$ satisfies the  preceding conditions and we can construct a smooth section  $s_{i,j}:B_{\tfrac{\rho}{2}}(x_j,M)\rightarrow M_i$ of $\psi_i$, such that
%(why there is  such s satisfies inequality below???)
% we choose this \rho infimum of the radius neighborhood on $M$ and $\rho$ up there.
\begin{eqnarray}\label{sec}
\frac{|(s_{i,j})_*(v)|}{|v|}< C
\end{eqnarray}
for each $v\in T B_{\tfrac{\rho}{2}}(x_j,M)$. Here $C$ is a constant independent of $i$. Hereafter we let $p_{i,j}=\psi^{-1}_i(x_j)$ and by $B(p_{i,j})$ we mean a ball centered at $p_{i,j}$ with radius $\rho$ in $T_{p_{i,j}}M_i$. See section $3$ in \cite{F88} and section $2$ in \cite{F89}. \\
\end{remark}
 Now we  show that $f_i$s are almost constant on the fibers of $M_i$. The following lemma is similar to Lemma $4.3$ in \cite{Fuk87}.  In the following lemma
 $(M_i,g_i)$ is a convergent sequence in $\cal{M}(n,D)$ such that $g_i$ satisfies only  (\ref{curv}) and $N$ is a compact Riemannian manifold.
 %%%%%%%%%%%%%%%%%%%%%%%%%%%%%%%%%%%%%%%%%%%%%%%%%%%%%%%%%%%%%%%%%%%%%%%%%%%%%%%%%%%%%%%%%%%%% %%%%%%%%%%%%
 %%%%%%%%%%%%%%%%%%%%%%%%%%%%%%%%%%%%%%%%%%%%%%%%%%%%%%%%%%%%%%%%%%%%%%%%%%%%%%%%%%%%%%%%%%%%%%
 %%%%%%%%%%%%%%%%%%%%%%%%%%%%%%%%%%%%%%%%%%%%%%%%%%%%%%%%%%%%%%%%%%%%%%%%%%%%%%%%%%%%%
\begin{lemme}\label{main1}
 Let $h_i:M_i\rightarrow I(N)\subset \mathbb R^q$ be  smooth maps which satisfy the Euler-Lagrange equation (\ref{EL}). Suppose $v_i\in T_p(M_i)$ satisfies $(\psi_i)_*(v_i)=0$,where $\psi_i$ is the fibration map and $v'_i,v''_i\in T_p(M_i)$ ($p\in B_{2\rho/3}(p_{i,j},M_i)$). Then we have
\begin{eqnarray}
|v_i\cdot h_i| &\leq &C_1\cdot \epsilon'_i \cdot|v_i|\cdot (\|\Delta h_i\|_{L^{\infty}}+\|h_i\|_{L^{\infty}}),\label{cons}\\
 |v'_i\cdot v''_i\cdot h_i| &\leq& C_2\cdot|v'_i|\cdot|v''_i|\cdot(\|\Delta h_i\|_{L^{\infty}}+\|h_i\|_{L^{\infty}}),\label{cons1}
\end{eqnarray}
where $C_1$ and $C_2$ are some constants independent of $i$ and $\epsilon'_i$ is a sequence converging to zero. Also $v_i \cdot h_i = dh_i(v_i)$ denotes the derivative of $h_i$ in the direction of $v_i$.
\end{lemme}
\begin{proof}
 We put $\Phi_{i,j}=\exp_{p_{i,j}}:{B(p_{i,j})}\rightarrow M_i$, $\tilde{g}_{i,j}={\Phi_{i,j}}_*(g_i)$ and $a=\Phi_{i,j}^{-1}(p)$. We also denote $h_i\circ\Phi_{i,j}$ by $h_{i,j}$.

  From the  Schauder estimates for elliptic equations (see Theorem \ref{schauder}) we have
 \begin{eqnarray}\label{sch}
 \|h_{i,j}\|_{C^{1,\alpha}}\leq C'\cdot(\|\Delta h_{i,j}\|_{L^{\infty}}+\|h_{i,j}\|_{L^{\infty}}),
 \end{eqnarray}
and hence
 \begin{eqnarray}\label{sch2}
 \|v'_i\cdot h_{i,j}\|_{C^{\alpha}}\leq C'\cdot(\|\Delta h_i\|_{L^{\infty}}+\|h_i\|_{L^{\infty}}),
 \end{eqnarray}
where $C'$ depends on the metric $\tilde{g}_{i,j}$.  Since $\Phi_{i,j}$ is an isometry, by the composition formula (see formula 1.4.1 in \cite{X96}), we have $\Delta h_{i,j}(x)=\Delta h_i(\Phi_{i,j}(x))$. Also from (\ref{sch}), and the fact that $\tilde{g}_{i,j}$ converges in $C^{\infty}$
  \begin{align*}
 \|\Pi(h_{i,j})(dh_{i,j},dh_{i,j})\|_{C^{\alpha}}\leq C''\cdot(\|\Delta h_i\|_{L^{\infty}}+\|h_i\|_{L^{\infty}}),
 \end{align*}
where $C''$ is a constant independent of $i$. By  equation (\ref{EL}), we have
\begin{align*}
\|\Delta h_{i,j}\|_{C^{\alpha}}\leq C''\cdot(\|\Delta h_i\|_{L^{\infty}}+\|h_i\|_{L^{\infty}}).
\end{align*}
Using  Schauder estimates for second derivative, we have
\begin{eqnarray}\label{sch1}
\|h_{i,j}\|_{C^{2,\alpha}}\leq C\cdot(\|\Delta h_i\|_{L^{\infty}}+\|h_i\|_{L^{\infty}}),
\end{eqnarray}
for some  $C$ independent of $i$ and  (\ref{cons1}) follows.
%%%%%%%%%%%%%%%%%%%%%%%%%%%%%%%%%%%%%%%%%%%%%%%%%%%%%%%%%%%%%%%%%%%%%%%%%%%%%%%%%%%%%%%%%%%%%%%%%%%%%%%%

Now we prove (\ref{cons}) by contradiction.  Assume $|v_i|=1$. Let $\sigma^i(t)=\exp^{F_i}_p(tv_i)$ be a geodesic in the fiber containing $p$, $F_i\subset M_i$  such that $\frac{d}{dt}|_{t=0}\sigma^i(t)=v_i$. For $0\leq t\leq \tfrac{\rho}{5}$ this curve has a lift $l^i(t)\subset B(p_{i,j})$ such that $\Phi_{i,j}(l^i(t))=\sigma^i(t)$. We have
\begin{align*}
d(\sigma_i(t),p)\leq \diam(F_i)\leq\epsilon_i.
\end{align*}
%  We define a curve $l^i:[0,\rho/5]\rightarrow B(p_{i,j})$ by $l^i(t)=\exp_a(tv_i)$.
% From  the fiber bundle theorem, we have
% \begin{align*}
%  d(\Phi_{i,j}\circ l^i(t),p)<\epsilon _i.
% \end{align*}
%for each $i$ and $t\in [0,\rho/5]$.
%%%%%%%%%%%%%%%%%%%%%%%%%%%%%%%%%%%%%%%%%%%%%%%%%%%%%%%%%%%%%%%%%%%%%%%%%%%%%%%%%%%%%%%%%%%%
  
  By contradiction we assume that there is subsequence of $h_i$ and a positive number $A$ such that
\begin{align*}
 |v_i\cdot h_{i,j}|> A \cdot (\|\Delta h_i\|_{L^{\infty}}+\|h_i\|_{L^{\infty}}).
\end{align*}
We know that
\begin{align*}
v_i\cdot h_{i}=v_i\cdot h_{i,j}= \left. \frac{d}{dt} \right|_{t = 0}h_{i,j}\circ l^i(t).
\end{align*}
 There exist $\beta>0$ and $\delta>0$ independent of $i$ such that for any $t<\delta$, we have
\begin{eqnarray}\label{proof}
|h_{i,j}\circ l^i(t)-h_{i,j}(a)|>\beta \cdot t\cdot (\|\Delta h_i\|_{L^{\infty}}+\|h_i\|_{L^{\infty}}).
\end{eqnarray}
To explain this, let $h_{i,j}\circ l^i(t)=q_{i,j}(t)$. We know from (\ref{sch1}) that
\begin{align*}
| \left. \frac{d}{dt} \right|_{t = 0} q'_{i,j}(t)|\leq C(\|\Delta h_i\|_{L^{\infty}}+\|h_i\|_{L^{\infty}}),
\end{align*}
so for some fixed $\delta$ and $0<t<\delta$ we have
\begin{align*}
|q'_{i,j}(t)-q'_{i,j}(0)|\leq C'\cdot t\cdot(\|\Delta h_i\|_{L^{\infty}}+\|h_i\|_{L^{\infty}}).
\end{align*}
On the other hand we have
\begin{align*}
 |q'_{i,j}(0)|> A \cdot (\|\Delta h_i\|_{L^{\infty}}+\|h_i\|_{L^{\infty}}),
\end{align*}
so for $\delta$ small enough  and $t<\delta$ we have
\begin{align*}
|q'_{i,j}(t)|> \beta \cdot (\|\Delta h_i\|_{L^{\infty}}+\|h_i\|_{L^{\infty}}).
\end{align*}
Therefore
\begin{align*}
|q_{i,j}(t)-q_{i,j}(0)|=|q'_{i,j}(\theta_i)\cdot t|>\beta \cdot t\cdot (\|\Delta h_i\|_{L^{\infty}}+\|h_i\|_{L^{\infty}}),
\end{align*}
from which (\ref{proof}) follows.\\

 %%%%%%%%%%%%%%%%%%%%%%%%%%%%%%%%%%%%%%%%%%%%%%%%%%%%%%%%%%%%%%%%
There exists $b\in B(p_{i,j})$, such that $d(a,b)<\epsilon _i$ and $\Phi_{i,j}(l_i(\delta'))=b$. For a fixed  $\delta'<\delta$ we have
\begin{align*}
|h_{i,j}(b)-h_{i,j}(a)|>\beta \cdot \delta'\cdot (\|\Delta h_i\|_{L^{\infty}}+\|h_i\|_{L^{\infty}}).
\end{align*}
 If we fix $\{\xi_k\}^{k=n}_{k=0}$ as a coordinate system at the point $a\in B(p_{i,j})$, for some $b'\in B(p_{i,j})$ we  have
\begin{eqnarray*}
\sum^{k=n}_{k=0}\frac{\partial h_{i,j}}{\partial \xi^k }> C\cdot \beta\cdot \tfrac{\delta'}{\epsilon_i}\cdot (\|\Delta h_i\|_{L^{\infty}}+\|h_i\|_{L^{\infty}}),
\end{eqnarray*}
and this contradicts (\ref{sch2}).
\end{proof}
%------------------------------------------%%%%%%%%%%%%%%%%%%%%%%%%%%%%%
Now we prove Proposition \ref{jadid}.
\begin{proof}[Proof of Proposition \ref{jadid}]
As we assumed  $\|e(f_i)\|_{L^{\infty}}<c$ and by the Euler-Lagrange equation and Corollary \ref{ef}, we have that $\|\Delta I\circ f_i\|_{L^{\infty}}$ is uniformly bounded. Moreover, $\| I\circ f_i\|_{L^{\infty}}$ is uniformly bounded. Using  (\ref{cons}), the maps $f_i$s are equicontinuous. By Lemma \ref{Rong}, there is a limit map $f:M\rightarrow N$ which is continuous.
%%%%%%%%%%%%%%%%%%%%%%%%%%%%%%%%%%%%%%%%%%%%%%%%%%%%%%%%%%%%%%%%%%

We consider the following maps on $M$,
 \begin{eqnarray}\label{tilde}
 \tilde{f_i}=\sum \beta_j\cdot (I\circ f_i)\circ s_{i,j},
 \end{eqnarray}
 where $\beta_j$ is an arbitrary $C^{\infty}$ partition of unity associated to $B_{\tfrac{\rho}{2}}(x_j,M)$,  $s_{i,j}$ is the section associated to $\psi_i$ as mentioned in Remark \ref{fukk}.  Along a subsequence, which we again denote by $f_i$, we have
\begin{align*}
\lim_{i\rightarrow \infty}f_i(s_{i,j}(x))=f(x)\quad\quad \text{for}~x\in B_{\tfrac{\rho}{2}}(x_j,M),
\end{align*}
 and also
\begin{align*}
\lim_{i\rightarrow \infty}\tilde{f_i}(x)=I\circ f(x)\quad\quad \text{for}~x\in B_{\tfrac{\rho}{2}}(x_j,M).
\end{align*}
%We claim that $\tilde{f_i}$ weakly converges to $i\circ f$ in ${\cal{H}}^{1}_{g}((M,\Phi \dvol),\mathbb R^q)$.
%The energy density of the maps $\tilde{f_i}$ and in fact .

Since the energy density of $f_i$ is bounded and also $s_{i,j}$ satisfies  (\ref{sec}), we have  $\|e(\tilde{f_i})\|_{L^{\infty}}$ is uniformly bounded.  By the same argument as above,  $\|\tilde{f_i}\|_{C^{1}}$ is bounded and $\tilde{f_i}$ converge uniformly to $I\circ f$.  Moreover $\psi_i$ has bounded second fundamental form (see Theorem 2.6 in \cite{CGF92}) and the same is true for $s_{i,j}$. So $\tilde{f_i}$ has bounded $C^2$-norm and there is a subsequence of $\tilde{f_i}$ which converges to $I\circ f$ in the $C^1$-topology.
%WARNING: The big mistake which I did, was considering $s_{i,j}$ isometry, while it was wrong.

Choose a local orthonormal frame $\{\bar{e}_{k}\}_{k=1}^{m}$ on $(M,g_i^{M})$. Denote  its horizontal lift on $(M_i,g_i)$ by $\{e_{k}\}_{k=1}^{m}$. Suppose $\{e_{t}\}_{t=m+1}^n$ is a local orthonormal frame field of the fiber  $F_i$ in $M_i$ such that $\{e_{k},e_{t}\}$ form a local orthonormal frame field in $M_i$  (note that we omit  the index $i$ for the orthonormal frame fields on $(M_i,g_i)$ and $(M,g_i^{M})$).
%%%%%%%%%%%%%%%%%%%%%%%%%%%%%%%%%%%%%%%%%%%%%%%%%%%%%%%%%%%%%%%%%%%%%%%%%%%%%%%%%%%%%%%%%%%%%%%%%%%%%%%%
Our aim is to show that $f$ is also weakly harmonic.
\begin{lemme}\label{lemma1}
We have
\begin{align*}
\lim_{i\rightarrow \infty}|\langle dI\circ f_i,d\eta_i\rangle(p)-\langle d\tilde{f_i},d{\eta}\rangle(\psi_i(p))|=0,
\end{align*}
where $\eta:M\rightarrow \mathbb R^q$, is a $C^{\infty}$-map  $\eta_i=\eta\circ \psi_i$, and $p$ in $M_i$.
\end{lemme}
\begin{proof}
By inequality (\ref{cons}),
\begin{align*}
|\langle dI\circ f_i,d\eta_i\rangle(p)-\sum_{k=1}^m\langle di\circ f_i(e_k),d\eta_i(e_k)\rangle(p)|\leq C_1\cdot\epsilon'_i
\end{align*}
for $i$ large enough where $C_1$ is a constant independent of $i$. Let $F_i$ denote the fiber containing $p$ and choose a point $q$ in $F_i$. By (\ref{cons1}), and  since $\diam(F_i)\leq \epsilon_i$
\begin{align*}
| dI\circ f_i(e_k)(p)- dI\circ f_i(e_k)(q)|\leq C_2\cdot \epsilon_i,
\end{align*}
 and so
 \begin{align*}
| dI\circ f_i(e_k)(p)- dI\circ f_i(e_k)(s_{i,j}\circ \psi_i(p))|\leq C_2\cdot \epsilon_i.
\end{align*}
% Let $\Pi:T_p(M_i)\rightarrow T_p(F^{\psi_i(p)})$ be the orthonormal projection.
Because $\psi_i\circ s_{i,j}=\Id$,  for $x\in M$ we have 
%\begin{eqnarray*}
%|\left(\Id-\Pi\right)\left(e_k(s_{i,j}(x))-{s_{i,j}}_*(\bar{e}_k(x))\right)|\leq o(\epsilon_i)
%\end{eqnarray*}
\begin{align*}
{\psi_i}_*\left(e_k(s_{i,j}(x))-{s_{i,j}}_*(\bar{e}_k(x))\right)=0.
\end{align*}
By inequality (\ref{sec}), we have
\begin{align*}
|e_k(s_{i,j}(x))-{s_{i,j}}_*(\bar{e}_k(x))|\leq C_3,
\end{align*}
for some constant $C_3$ and therefore by (\ref{cons}),
\begin{align*}
|dI\circ f_i(e_k)(p)- d(I\circ f_i)\circ{ s_{i,j}}_*(\bar{e}_k)(\psi_i(p))|\leq C_4\cdot \epsilon_i.
\end{align*}
From the convergence of  $f_i\circ s_{i,j}$  to $f$, we have
\begin{align*}
\lim_{i\rightarrow\infty}|\sum d\beta_j \cdot (I\circ f_i)\circ s_{i,j}-\sum d\beta_j\cdot (I\circ f)|=0,
\end{align*}
 So
%d\beta_j \cdot f_i\circ s_{i,j}
\begin{align*}
\lim_{i\rightarrow \infty}|d\tilde{f_i}-\sum \beta_j \cdot d((I\circ f_i)\circ s_{i,j})|=0.
\end{align*}
 Since $\sum_j\beta_j=1$  we finally have
\begin{align*}
\lim_{i\rightarrow}|\langle dI\circ f_i,d\eta_i\rangle(p)-\langle d\tilde{f_i},d{\eta}\rangle(\psi_i(p))|=0.
\end{align*}
\end{proof}
%As we explained above, $\tilde{f_i}$ converges to $i\circ f$ in the $C^{1}$-topology:
%\begin{eqnarray*}
%|d\tilde{f_i}-di\circ f|\leq \delta_i
%\end{eqnarray*}
%for $i$ large enough and $\delta_i$ goes to zero.
% \begin{align*}
% \lim_{i\rightarrow \infty}|\sum d\beta_j{({\bar{e}_k})}\cdot df_i({s_{i,j}}_*({\bar{e}_k}))|=0
% \end{align*}
%finally we have
%\begin{eqnarray*}
%|\sum \beta_j \cdot d((i\circ f_i)\circ s_{i,j})-di\circ f|\leq\delta_i
%\end{eqnarray*}
%%%%%%%%%%%%%%%%%%%%%%%%%%%%%%%%%%%%%%%%%%%%%%%%%%%%%%%%%%%%%%%%%%%%%%%%%%%%%%%%%%%%%%%%%%%%%%%%%%
%%%%%%%%%%%%%%%%%%%%%%%%%%%%%%%%%%%%%%%%%%%%%%%%%%%%%%%%%%%%%%%%%%%%%%%%%%%%%%%%%%%%%%%%%%%%5%%%%%%%
\begin{lemme}\label{lemma2}
We have
\begin{align*}
\lim_{\i\rightarrow\infty}\left|\Pi(f_i)(p)(dI\circ f_i,dI\circ f_i)-\Pi(\tilde{f_i})(\psi_i(p))(d\tilde{f_i},d\tilde{f_i})\right|=0.
\end{align*}
\end{lemme}
\begin{proof}
By the  proof of the above lemma, we have
\begin{align*}
\lim_{i\rightarrow \infty}|df_i(p)-d\tilde{f_i}(\psi_i(p))|=0.
\end{align*}
 By the same argument as in Lemma \ref{lemma1} we can conclude
\begin{eqnarray*}
&\left|\Pi(f_i)(p)(dI\circ f_i,dI\circ f_i)-\Pi(\tilde{f_i})(\psi_i(p))(d\tilde{f_i},d\tilde{f_i})\right|\\
&\leq C\cdot\left|df_i(p)-d\tilde{f_i}(\psi_i(p))\right|.
\end{eqnarray*}
\end{proof}
%%%%%%%%%%%%%%%%%%%%%%%%%%%%%%%%%%%%%%%%%%%%%%%%%%%%%%%%%
 The map $\tilde{f_i}:(M, g_i^M, \dvol_{g_i^M})\rightarrow \mathbb R^q$ converges in $C^1$ to the map $I\circ f$, and $\Phi_i$ converges to $\Phi$ in the $C^{\infty}$ topology. Also $(M,g_i^M)$ converges to $(M,g)$ in $\mathcal{M}(n,D,v)$.  Therefore we have
\begin{eqnarray*}\label{limit}
\left|\int_{M}\Xi_{g_i^M}(\eta,\tilde{f_i})~\Phi_i\dvol_{g_i^M}-\int_M\Xi(\eta,f)~\Phi\dvol_{g}\right|\leq C\cdot\epsilon_i,
\end{eqnarray*}
where $\Xi(\cdot,\cdot)$ is defined by (\ref{khi}). By Lemma \ref{lemma1} and \ref{lemma2}, we have
\begin{align*}
\lim_{i\rightarrow\infty}\left|\int_{M_i}\Xi_{g_i}(\eta_i,f_i)\tfrac{\dvol_{M_i}}{\vol(M_i)}-\int_{M}\Xi_{g_i^M}(\eta,\tilde{f_i}){\psi_i}_*\left(\tfrac{\dvol_{M_i}}{\vol(M_i)}
\right)\right|=0.
\end{align*}
It follows that
\begin{eqnarray}\label{finish}
\lim_{i\rightarrow \infty}\int_{M_i}\Xi_{g_i}(\eta_i,f_i)\tfrac{\dvol_{M_i}}{\vol(M_i)}=\int_{M}\Xi_g(\eta,f)~\Phi\dvol_M.
\end{eqnarray}
Therefore  $f$ is weakly harmonic and since it is continuous, it is also a smooth harmonic map.
\end{proof}
%%%%%%%%%%%%%%%%%%%%%%%%%%%%%%%%%%%%%%%%%%%%%%%%%%%%%%%%%%
Now we prove Case II without considering Assumption \ref{assumption}.
 \begin{proof}[Proof of Proposition \ref{CaseII}]
 By Remark \ref{explain} we can obtain a $C^1$-close metric $g_i(\epsilon)$  to $g_i$ which satisfies  (\ref{curv}) and
such that the map $\psi_i:(M_i,g_i(\epsilon))\rightarrow (M,{\psi_i}_*(g_i(\epsilon)))$ is a Riemannian submersion.

For small $\epsilon$, let $M(\epsilon)$ be the Gromov-Hausdorff limit of a subsequence  of $(M_i,g_i(\epsilon))$. By Lemma 2.3 in \cite{F88},   $(M_i,g_i(\epsilon))$ and $(M(\epsilon),g(\epsilon))$ converge  to $(M_i,g_i)$ and $(M,g)$  in ${\cal{M}}(n,D,v)$ respectively.

 The map $f_i:(M_i,g_i)\rightarrow (N,h)$ is harmonic and since $g_i(\epsilon)$ is $C^1$-close to $g$, we have
 \begin{align*}
 |\Xi_{g_i}(f_i,\eta_i)-\Xi_{g_i(\epsilon)}(f_i,\eta_i)|\leq C\cdot\epsilon .
 \end{align*}
%Let $f(\epsilon):(M,g(\epsilon))\rightarrow N$ be the limit of sequence of maps $f_i:(M_i,g_i(\epsilon))\rightarrow N$.
By (\ref{finish}), we have
\begin{align*}
\lim_{i\rightarrow \infty}\left|\int_{M_i}\Xi_{g_i(\epsilon)}(f_i,\eta_i)\tfrac{\dvol_{(M_i,g_i(\epsilon))}}{\vol((M_i,g_i(\epsilon)))}
-\int_{M(\epsilon)}\Xi_{g(\epsilon)}(f,\eta)\cdot\Phi(\epsilon) \dvol_{M(\epsilon)}\right|=0,
\end{align*}
and finally since $g(\epsilon)$ converges to $g$ in the $C^{1,\alpha}$-topology, we have the desired result.
\end{proof}
%%%%%%%%%%%%%%%%%%%%%%%%%%%%%%%%%%%%%%%%%%%%%%%%%%%%%%%%
\subsection{Case III: Collapsing to a singular space.}\label{caseiii}
Now we are going to investigate the general case when the sequence converges to a singular space.
 This means that $(M_i,g_i)$  in  ${\cal{M}}(n,D)$ converges to some metric space $(X,d)$. First we recall the following remark from \cite{Fuk87}.
\begin{remark}[Fukaya \cite{Fuk87}, \S7] \normalfont Let $Y$ be a Riemannian manifold on which $O(n)$ acts by isometry, and let $\theta:Y\rightarrow [0, \infty)$ be an $O(n)$-invariant smooth function. Put $X = Y/O(n)$. Let $p: Y\rightarrow X$ be the natural projection, $\bar{\theta}: X\rightarrow [0,\infty)$ the function induced from $\theta$, and $S(X)$ the set of all singular points of $X$. The set $S(X)\subset X$ has a well defined normal bundle on the codimension $2$ strata ($X=Y/O(n)$ is a Riemannian polyhedron and $S(X)$ is a subset of  the $(n-2)$-skeleton of $X$). Set
\begin{align*}
\lip(X,S(X))=\{u\in \lip(X)~|~ v\cdot u=0 ~ \text{if}~ v~\text{is perpendicular to}~ S(X)\}.
\end{align*}
 Define $Q_1: \lip(Y) \times \lip(Y)\rightarrow[0, \infty)$ and $Q_2: \lip( X, S(X))
\times \lip(X,S(X))\rightarrow [0, 1)$ by
\begin{eqnarray*}
Q_1(\tilde{k},\tilde{h} )&=& \int_Y \theta\cdot \langle\nabla\tilde{ k},\nabla \tilde{h}\rangle~ \dvol_Y,\\
Q_2(k,h)&=&\int_X  \bar{\theta}\cdot\langle\nabla k, \nabla h\rangle ~d\mu_g.
\end{eqnarray*}
It is easy to see that $f\circ p\in \lip(Y)$ for each $f$ contained in $\lip(X,S(X))$. Define
$p^*: \lip(X,S(X))\rightarrow \lip(Y)$ by $p^*(f)=f\circ p$. Let $\lip_{O(n)}(Y)$ be the set of all $O(n)$-invariant elements of $\lip(Y)$. Then, we can easily prove the following
\begin{lemme}\label{fuk7}
$ p^*$ is a bijection between $\lip(X, S(X))$ and $\lip_{O(n)}(Y)$. For elements  $f$  and $k$ of $\lip(X,S(X))$, we have
\begin{eqnarray}\label{Qone}
Q_1 (f, k)=Q_ 2 (p^* (f), p^* (k)),
\end{eqnarray}
and
\begin{eqnarray}\label{Qtwo}
\int_Y \theta\cdot p^*(f)p^*(k) ~\dvol_Y=\int_X \bar{\theta}\cdot fk ~d\mu_g.
\end{eqnarray}
\end{lemme}
\end{remark}
Now we prove the main theorem of this paper.
\begin{proof}[Proof of Theorem \ref{main}]
We denote by  $(Y,g,\Phi_Y\dvol_Y)$ the limit space of the frame bundles  over $M_i$, and by $(X,d,\nu)$ the limit space of $M_i$   with respect to the measured Gromov-Hausdorff topology. We know $(X,\nu)=(Y,\Phi_Y\dvol_Y)/O(n)$ (see Subection \ref{densityfunction} ).  The projection   $p_i:(F(M_i),\tilde{g}_i)\rightarrow (M_i,g_i)$ is a Riemannian submersion with totally geodesic fibers. So using the reduction formula the map $\bar {f}_i=f_i\circ p_i$ is harmonic on $F(M_i)$ and it is invariant under the action of $O(n)$. Furthermore $\|e_{\tilde{g}_i}(\bar{f}_i)\|_{\infty}$ is bounded ($p_i$ is a Riemannian submersion). Using Case II, $\bar{f}_i$ converge to some map $\bar{f}$ on $(Y,g,\Phi_Y\dvol_Y)$. The map $\bar{f}$  satisfies
\begin{eqnarray*}\label{ahmagh}
\int_{Y}\Xi_{g}(\bar{f},\eta)~\Phi_Y\dvol_Y=0,
\end{eqnarray*}
 where $\eta$ is a test function. The map $\bar{f}$ is also $O(n)$ invariant and continuous. Consider a quotient map $f$ such that $\bar{f}= p^*(f)$. First we show that $f$ is in ${\cal{H}}^1((X,\nu),N)$. By the argument in Case II, $\bar{f}$ is in ${\cal{H}}^1((Y,\Phi_Y \dvol_Y),N)$ and so by  equation (\ref{Qone}), $f$ has finite energy. 
  Now we show that $f$ is weakly harmonic on $(X,\nu)$. By equation (\ref{Qone}), for $\eta$ in $\lip(X,S(X))$
\begin{align*}
\int_Y \langle \nabla I\circ\bar{f},\nabla p^*(\eta)\rangle~ \Phi_Y\dvol_Y=\int_X \langle \nabla I\circ {f},\nabla \eta\rangle~\Phi_X d\mu_g.
\end{align*}
Furthermore
\begin{eqnarray*}
\int_Y \langle \Pi(\bar{f})(\nabla^{g}(I\circ\bar{f}),\nabla ^{g}(I\circ\bar{f})),p^*(\eta)\rangle~ \Phi_Y\dvol_Y\\
=\int_X \langle \Pi({f})(\nabla(I\circ{f}),\nabla(I\circ{f})),\eta\rangle~ \Phi_X d\mu_g,
\end{eqnarray*}
and since $\Phi_Y=p^* (\Phi_X)$
\begin{align*}
\int_{Y}\Xi_{g}(\bar{f},p^*(\eta))~\Phi_Y\dvol_Y=
\int_{X}\Xi({f},\eta)~\Phi_X d\mu_g,
\end{align*}
which shows that  $f:X\rightarrow N$ is a weakly harmonic map.%?????
\end{proof}
%%%%%%%%%%%%%%%%%%%%%%%%%%%%%%%%%%%%%%%%%%%%%%%%%%%%%%%%%%%%%%%%%%%%%%%%%%%%%%%%%%%%%%%%%%%%%%%%%%%%%%%%%%%%%%%%%%%%%%%%%%%%
\section{Appendix: Convergence of tension field.}
 In this section we study convergence of the tension fields of the maps $f_i$, $\tau(f_i)$, under the assumptions of Proposition \ref{jadid}.\\\\
  Assume $(M_i,g_i)$, $f_i$, $N$ to be as in Proposition \ref{jadid}.  Moreover consider the following assumption
   \begin{asum}
 The section $s_{i,j}$ is almost harmonic,
\begin{eqnarray}\label{alhar}
|\tau(s_{i,j})|\leq C\cdot\epsilon''_i,
\end{eqnarray}
and also
\begin{eqnarray}\label{geo}
|\nabla _{\bar{X}}ds_{i,j}(X)|\leq C\cdot\epsilon''_i,
\end{eqnarray}
where $X$ is a smooth vector field on $M$ and $\bar{X}$ is its horizontal lift and $\epsilon''_i$ is a sequence which converges to zero.
 \end{asum}

Using Assumption \ref{assumption} and by Theorem \ref{reduction} we have
   \begin{eqnarray}\label{decom}
\tau(f_i)&=&({\nabla_{e_k}}df_i)e_k+({\nabla_{e_t}}df_i)e_t\\
&=&({\nabla_{e_k}}df_i)e_k+\nabla_{{f_i}_*(e_t)}{{f_i}_*(e_t)}{\nonumber}\\
&-&{f_i}_*(\nabla_{e_t}{e_t})^{H}- {f_i}_*(\nabla_{e_t}{e_t})^V{\nonumber}\\
&=&({\nabla_{e_k}}df_i)e_k-{f_i}_*(\h_i)+\tau ({f_i}^{\bot}){\nonumber}
\end{eqnarray}
  where $\{e_k,e_t\}$ and $\bar{e}_k$ are as in the proof of Proposition \ref{jadid}, ${f_i}^{\bot}$ denotes the restriction of $f_i$ to the fibers $F_i$, and $\h_i$ is the mean curvature vector of the submanifold $F_i$.\\\\
  %%%%%%%%%%%%%%%%%%%%%%%%%%%%%%%%%%%%%%%%%%%%%%%%%
We investigate how each term of the equation above behaves as $f_i$ converges to $f$.
 \begin{lemme}
We have
\begin{eqnarray}\label{first}
\lim_{i\rightarrow \infty}\left|d I({\nabla_{e_k}}d f_i)e_k(p)-
\left(\Delta_{g_i^M} \tilde{f_i}-\Pi(\tilde{f_i})(d\tilde{f_i},d\tilde{f_i})\right)(\psi_i(p))\right|=0.
\end{eqnarray}
\end{lemme}
\begin{proof}
By the discussion in the proof of Proposition \ref{jadid}, we know that $\tilde{f_i}$ converges to $f$ in the $C^1$-topology. Using the composition formula we have
\begin{align*}
 dI(Bf_i(X_1,X_2))=B(I\circ f_i)(X_1,X_2)-B(\pi_N)(d (I\circ f_i)(X_1),d (I\circ f_i)(X_2)),
\end{align*}
 and so for $k=1,\ldots,n$,
\begin{align*}
 dI(({\nabla_{e_k}}df_i)e_k)=({\nabla_{e_k}}d( I\circ f_i))e_k-B(\pi_N)(d(I\circ f_i)(e_k),d(I\circ f_i)(e_k)).
\end{align*}
 First we show that
 \begin{align*}
\lim_{i\rightarrow \infty}|{\nabla_{e_k}}d(I\circ f_i)e_k(p)-
\Delta_{g_i^M} \tilde{f_i}(\psi_i(p))|=0.
\end{align*}
   By  definition of $\tilde{f_i}$,
 \begin{eqnarray*}\label{jalal}
 (\nabla_{{\bar{e}_k}} d\tilde{f_i}){\bar{e}_k}&=&\sum \big (d\beta_j{({\bar{e}_k})}\cdot d f_i({s_{i,j}}_*({\bar{e}_k}))\\
 &+&\beta_j\cdot(\nabla_{{\bar{e}_k}}d(f_i\circ s_{i,j})){\bar{e}_k}+\triangle \beta_j\cdot  f_i\circ s_{i,j}\big).
 \end{eqnarray*}
and again by the composition formula 
\begin{eqnarray}\label{compp}
\tau(f_i\circ s_{i,j})=B_{{s_{i,j}}_*(\bar{e}_k),{s_{i,j}}_*(\bar{e}_k)}{f_i}+df_i(\tau(s_{i,j})).
\end{eqnarray}
Since $f_i\circ s_{i,j}$ converges in $C^1$ to $f$
 \begin{eqnarray*}
& \lim_{i\rightarrow \infty}|\sum d\beta_j{({\bar{e}_k})}\cdot d f_i({s_{i,j}}_*({\bar{e}_k}))|=0,\\
&\lim_{i\rightarrow \infty}\sum\Delta \beta_j\cdot  f_i\circ s_{i,j}(x)= \sum\triangle \beta_j\cdot f(x)=0.
 \end{eqnarray*}
Also, ${\psi_i}_*(e_k-{s_{i,j}}_*(\bar{e}_k))=0$ and so $e_k-{s_{i,j}}_*(\bar{e}_k)$ is vertical. On the other hand
\begin{align*}
|e_k-{s_{i,j}}_*(\bar{e}_k)|\leq\epsilon_i.
\end{align*}
By inequality (\ref{cons}) and almost harmonicity of $s_{i,j}$ (\ref{alhar}), the second term on the right hand side of (\ref{compp}) converges to zero. Again by inequality (\ref{cons1}) and  (\ref{geo}), we have
\begin{eqnarray*}
&\lim_{i\rightarrow \infty}|(\nabla_{e_k}{d f_i})(e_k-{s_{i,j}}_*(\bar{e}_k))|=0,\\
&\lim_{i\rightarrow \infty}|(\nabla_ {(e_k-{s_{i,j}}_*(\bar{e}_k))}d{f_i})e_k|=0.
\end{eqnarray*}
Finally
\begin{align*}
\lim_{i\rightarrow\infty}|(\nabla_{e_k} d(I\circ f_i))e_k(p)-(\nabla_{\bar{e}_k}d\tilde{f_i})\bar{e}_k(\psi_(p))|=0.
\end{align*}
We have the same for the second term
\begin{align*}
\lim_{i\rightarrow\infty}|\Pi(f_i)(p)(df_i,df_i)-\Pi(\tilde{f_i})(\psi_i(p))(d\tilde{f_i},d\tilde{f_i})|=0.
\end{align*}
\end{proof}
By the above lemma and  ${\psi_i}_*(\tfrac{\dvol_{M_i}}{\vol(M_i)})=\Phi_i\dvol_{M}$, we have
\begin{align*}
\lim_{i\rightarrow\infty}\left|\int_{M_i}\langle dI(({\nabla_{e_k}}d f_i)e_k),\eta_i\rangle~\tfrac{\dvol_{M_i}}{\vol(M_i)}-
\int_M \langle\Delta^{g_i^M} \tilde{f_i}-\Pi(\tilde{f_i})(d\tilde{f_i},d\tilde{f_i}),\eta\rangle~\Phi_i\dvol_{g_i^M}\right|=0,
\end{align*}
and we conclude
\begin{eqnarray}\label{firstt}
&\lim_{i\rightarrow\infty}\int_{M_i}\langle dI(({\nabla_{e_k}}d f_i)e_k),\eta_i\rangle~\tfrac{\dvol_{M_i}}{\vol(M_i)}\nonumber\\
&=\int_{M}\left[\langle df,d\eta\rangle+\langle df(\nabla\ln \Phi)-\Pi(f)(df,df),\eta\rangle\right]~ \Phi\dvol_M.
\end{eqnarray}
Here $\eta$ is a test map on $M$ and $\eta_i=\eta\circ\psi_i$.
%%%%%%%%%%%%%%%%%%%%%%%%%%%%%%%%%%%%%%%%%%%%%%%%%%%%%%%%%%%%%%%%%%%
Now we will consider the second and third terms in the decomposition of $\tau(f_i)$.
\begin{lemme} \label{hamechi}With the same assumptions as above 
\begin{enumerate}[i.]
\item $\lim_{i\rightarrow \infty} \int_{M_i}\langle df_i(\h_i),\eta_i\rangle~ \tfrac{\dvol_{M_i}}{\vol(M_i)}=-\int_{M}\langle df(\nabla \ln\Phi),\eta\rangle~ \Phi \dvol_M.$
\item $\lim _{i\rightarrow \infty}\|\tau({f_i}^{\bot})\|=0.$
\end{enumerate}
Here $\h_i$ denotes the mean curvature vector of the fibers $F_i^x=\psi^{-1}_i(x)$. 
\end{lemme}

{Before we prove  Lemma \ref{hamechi}, we prove the following lemma which we need for the proof of part i.}
\begin{lemme}\label{two}
We have
\begin{align}
\int_M\eta d\ln\Phi(X)~\Phi\dvol_M=-\lim_{i\rightarrow \infty}\int_{M_i} \eta\langle X,\h_i\rangle ~ \tfrac{\dvol_{M_i}}{\vol(M_i)}.
\end{align}
\end{lemme}
\begin{proof} Suppose $X$ is a smooth vector field on $M$ and $X_i$ its horizontal lift on $M_i$. The flow $\theta_t^i $ of $X_i$ sends fibers to fibers diffeomorphically. By the first variation formula
\begin{eqnarray}\label{mean}
 \left. \frac{d}{dt} \right|_{t = 0} {\theta_t^i}^*(\dvol_{F^x_i})=-\int_{F^x_i}\langle X_i,\h^x_i \rangle ~ \dvol_{F^x_i}.
\end{eqnarray}
Also
\begin{align*}
\Phi_i(x)=\tfrac{\vol(\psi_i^{-1}(x))}{\vol(M_i)}.
\end{align*}
and by (\ref{mean}),
\begin{align*}
d\Phi_i(X)(x)=-\int_{F^x_i}\langle X_i,\h^x_i \rangle ~ \tfrac{\dvol_{F^x_i}}{\vol(M_i)},
\end{align*}
For an arbitrary $\eta$ in $C^{\infty}({M})$, we prove
\begin{eqnarray}\label{ahmah}
\int_{M} \eta d\Phi_i(X)~ \dvol_{g_i^M}=-\int_{M_i}\eta_i\langle X_i,\h_i \rangle ~ \tfrac{\dvol_{M_i}}{\vol(M_i)}.
\end{eqnarray}
If we consider $(U_{\gamma},h_{\gamma})$ as a local trivialization of the  fibration $\psi_i$, then
\begin{align*}
\int_{M} \chi_{U_{\gamma}} d\Phi_i(X)~ \dvol^{g_i^M}=-\int_{U_{\gamma}}\int_{F^x_i}\chi_{U_{\gamma}}\langle X_i,\h^x_i\rangle ~ \tfrac{\dvol_{F^x_i}}{\vol(M_i)}\dvol_{g_i^M},
\end{align*}
and so
\begin{align*}
\int_{M} \chi_{U_{\gamma}} d\Phi_i(X)~ \dvol^{g_i^M}_{M}=-\int_{\psi_i^{-1}( U_{\gamma})}\langle X_i,\h_i \rangle ~ \tfrac{\dvol_{M_i}}{\vol(M_i)},
\end{align*}
where $\chi_{U_{\gamma}}$ denotes the characteristic function on $U_{\gamma}$ and so we have (\ref{ahmah}).
%using  lemma 3.8 in \cite{F89} and
The functions $\Phi_i$ goes to $\Phi$ in $C^{\infty}$ and also $\dvol^{g_i^M}$ goes to $\dvol_M$ as $i$ goes to infinity. Letting $i$ go to $\infty$ on the both sides of  (\ref{ahmah}) and by the definition of weak derivatives
%\begin{eqnarray*}
%\int_{M} \eta d\Phi(X)~ \dvol^{g_i^M}_{M}=-\int_{M} \eta \lim_{i\rightarrow \infty}\langle X,{\psi_i}_*(\h^x_i) \rangle ~ \Phi(x) \dvol_{M}
%\end{eqnarray*}
\begin{align*}
\int_M\eta d\ln\Phi(X)~\Phi\dvol_M=-\lim_{i\rightarrow \infty}\int_{M_i} \eta\langle X,\h_i\rangle ~ \tfrac{\dvol_{M_i}}{\vol(M_i)}.
\end{align*}
\end{proof}
\begin{proof}[Proof of Lamma \ref{hamechi}]
Part i follows directly from Lemma \ref{two}.
%%%%%%%%%%%%%%%%%%%%%%%%%%%%%%%%%%%%%%%%%%%%%%%%%%%%%%%%%%%%%%%%%%%%%%%%%%%%%%%%%%%%%%%%%%%%%%%

To prove part  ii consider
\begin{eqnarray*}
\tau({f_i}^{\bot})=\nabla_{{f_i}_*(e_t)}{{f_i}_*(e_t)}-{f_i}_*(\nabla_{e_t}{e_t})^{V}.
\end{eqnarray*}
From (\ref{cons}) and (\ref{cons1})
\begin{eqnarray*}
|\nabla_{{f_i}_*(e_t)}{{f_i}_*(e_t)}|&<&C\cdot{\epsilon'_i},\\
\|{f_i}_*(\nabla_{e_t}{e_t})^{V}\|_{L^{\infty}}&<&C\cdot{\epsilon'_i}|(\nabla_{e_t}{e_t})^V|,
\end{eqnarray*}
where $C$ is a constant independent of $i.$ It follows that
\begin{align*}
\lim _{i\rightarrow \infty}\|\tau({f_i}^{\bot})\|=0.
\end{align*}

\end{proof}

%\bibliographystyle{alpha}
%\bibliography{biblio-2}

\def\cprime{$'$} \def\cprime{$'$} \def\cprime{$'$} \def\cprime{$'$}

\end{document}